\documentclass{amsart}

\usepackage[utf8]{inputenc}
\usepackage{microtype}

\usepackage[backref,bookmarks=true,colorlinks=true,pdfstartview=FitV,linkcolor=blue, citecolor=red,urlcolor=green]{hyperref}  

\usepackage{mathrsfs}

\usepackage{amssymb,amsmath,amscd}
\usepackage{framed}

 \usepackage{stmaryrd}
\usepackage{url}

\usepackage[numbers]{natbib} 

\usepackage{subfig}
\usepackage{graphicx}
\usepackage{lpic}
\usepackage{etoolbox}
\usepackage{bbm} 

\makeatletter
\patchcmd{\lp@dimens}
{\epsfig{figure=\lp@epsfile}}
{\includegraphics{\lp@epsfile}}
{}{}
\patchcmd{\lp@dimens}
{\epsfig{figure=\lp@epsfile, width=\lp@tmpx, height=\lp@tmpy}}
{\includegraphics[width=\lp@tmpx,height=\lp@tmpy]{\lp@epsfile}}
{}{}
\makeatother

\usepackage[export]{adjustbox}

\usepackage{xcolor}

 \newcommand{\red}[1]{\textcolor{black}{#1}}

\numberwithin{equation}{section}

\newcommand\bR{{\mathbb{R}}}

\newcommand\bZ{{\mathbb Z}}

\newcommand\vecu{\mathbbm{u}}

\newcommand{\vecx}{\mathbbm{x}}

\newcommand{\vecv}{\mathbbm{v}}
\newcommand{\vecw}{\mathbbm{w}}

\newcommand\Hom{{\rm Hom}}
\newcommand\dev{{\bf dev}}
\newcommand\SI{{\mathbb{S}}}

\newcommand\clo{{\rm Cl}}
\newcommand\bdd{{\mathbf{d}}}

\newcommand\ra{\rightarrow}

\newcommand\che{\check}
\newcommand\emp{\emptyset}
\newcommand\eps{\epsilon}

\newcommand\Idd{{\rm I}}

\newcommand\Isom{{\mathbf{Isom}}}

\newcommand\SO{{\mathsf{SO}}}

\newcommand\PGL{{\mathsf{PGL}}}

\newcommand\Ss{{\mathbb{S}}}

\newcommand\diam{{\mathrm{diam}  } }

\newcommand{\deltxt}[1]{}
\newcommand{\KFAR}[1]{}

\newcommand\Lspace{\mathsf E}
\newcommand\Uu{\mathsf U}

\newcommand\Ker{{\mathrm{Ker}}}

\newcommand\PSO{{\mathsf{PSO}}}

\newcommand\Bs{{\mathsf{B}}}

\newcommand{\va}{\mathbf{a}}
\newcommand{\vb}{\mathbf{b}}
\newcommand{\vc}{\mathbf{c}}

\newcommand{\bg}{\mathsf{g}}

	\newtheorem{theorem}{Theorem}[section]
\newtheorem{proposition}[theorem]{Proposition}
\newtheorem{lemma}[theorem]{Lemma}
\newtheorem{corollary}[theorem]{Corollary}

\theoremstyle{definition}
\newtheorem{definition}{Definition}[section]

\theoremstyle{remark}

\begin{document}



\title{Deformations of Margulis space-times with parabolics}


\author{Suhyoung Choi} 
\address{Department of Mathematical Sciences, KAIST,
		305-701, Daejeon, Republic of Korea}
\email{schoi@math.kaist.ac.kr}

\thanks{This work was supported by 
the National Research Foundation of Korea(NRF) grant funded by the Korea government 
(MEST) (No. NRF-2022R1A2C3003162).}

\begin{abstract} 
	Let $\Lspace$ be a flat Lorentzian space of signature $(2, 1)$. 
	A Margulis space-time  is a noncompact complete Lorentz flat $3$-manifold $\Lspace/\Gamma$ with a free isometry group $\Gamma$ of rank $\bg, \bg \geq 2$. 
	We consider the case when $\Gamma$ contains a parabolic element.
We show that sufficiently small deformations of $\Gamma$ still act properly on 
$\Lspace$. 
We use our previous work showing that $\Lspace/\Gamma$ can be compactified relative to a union of solid tori and some old idea of Carri\`ere in his famous work. 
We will show that the there is also a decomposition of $\Lspace/\Gamma$ by crooked planes that are disjoint and embedded in a generalized sense. These can be perturbed so that $\Lspace/\Gamma$ decomposes into cells. This partially affirms the conjecture of Charette-Drumm-Goldman.

\end{abstract} 
	
	







\date{\today}

\maketitle 


\section{Introduction} 
Let $\Lspace$ denote $\bR^{2, 1}$ with the standard coordinate systems, 
the standard Lorentzian inner product, and the standard orientation. 
Let $\Gamma$ be a discrete free subgroup of $\Isom(\Lspace)$ of rank $\geq 2$ acting properly discontinuously on $\Lspace$. 
Then $\Lspace/\Gamma$ is said to be a {\em Margulis space-time}. 

Every element of $\Isom(\Lspace)$ is of form $x \mapsto Ax + b$ for $A \in \SO(2, 1)$ and $b\in \bR^{2, 1}$. 
There is a homomorphism $\mathcal{L}:\Isom(\Lspace) \ra \SO(2, 1)$ taking the linear part of the affine transformations. 
Denote by $\SO(2, 1)^+$ the subgroup of elements preserving the upper cone. 
We denote by $\Isom^+(\Lspace)$ the inverse image of $\SO(2, 1)^+$ in $\Isom(\Lspace)$. 


We \cite{CDG22} proved earlier that every Margulis space-time $M$ with parabolics 
can be compactified relative to solid tori. That is we can add a totally geodesic real projective surfaces and remove a union of disjoint solid tori to obtain a compact $3$-manifold whose interior is homemorphic to $M$.

We can consider $\SO(2, 1)$ to act on a hyperbolic space using the Klein model by projectivizing. 
Note by the standard complete 
hyperbolic surface theory due to Fricke that for any free subgroup 
$\Gamma$ of rank $n$, $n\geq 1$, of $\Isom(\Lspace)$, 
we can deform it to one without 
parabolics by considering the quotient the hyperbolic space under 
$\Gamma$ as a complete hyperbolic surface and deforming. 
Actually, the subspace where there are parabolics is an analytic subspace of 
codimension $i$ where $i$ is the number of parabolic cusps. 
Since parabolics can only occur for elements parallel to boundary, we need to 
remove some finitely many subspaces of the Fricke space to obtain 
Fuchsian groups without parabolics. 
(See Lemma \ref{lem:deformnop}.)  

We prove the following connectivity: 
\begin{theorem}[Deformation of Margulis space-times] \label{thm:main} 
Let $\Gamma$ be a discrete free group in $\Isom(\Lspace)$ of rank $n \geq 2$ with a parabolic.  
Suppose that $\Gamma$ acts properly on $\Lspace$ so that 
$\Lspace/\Gamma$ is a manifold.  
Let $g_1,\dots, g_n$ denote the free generators. 
Then 
there exists a neighborhood $N$ of $(g_1, \dots, g_n) \in \Isom(\Lspace)^n$, 
the group $\Gamma_\mu$ 
generated by each $\mu=(g'_1, \dots, g'_n)\in N$ without parabolics
is a free group of rank $n$ acting properly discontinuously and freely on $\Lspace$. 
\end{theorem} 


Let $\Lspace$ have the standard coordinates $x, y, z$ so that the bilinear form is given 
by $x^2+ y^2 - z^2$. 
Given a space-like $\vecu \in \bR^{2,1}$, the {\em null frame} associated to $\vecu$ is 
the basis $(\vecu, \vecu^-, \vecu^+)$ in the standard orientation 
where $\vecu^+, \vecu^-$ are null vectors such that 
\begin{itemize} 
\item $\vecu^\perp = \langle \vecu^-, \vecu^+ \rangle$.
\item the third coordinates are $1$. 
\end{itemize}

Given a nonzero null-vector $\vecv$ in the closure of the positive cone, we define 
$P(\vecv):= \{ \vecx| \vecx \cdot \vecv = 0\}$. 
Note $\bR \vecv \subset P(\vecv)$ and 
$P(\vecv) - \bR \vecv$ has two components. 
Choose any $\vecw$ so that $\vecw$ and $\vecv$ lies on the closure of the positive 
cone. Then we define a {\em wing} at $\vecv$ to be 
\[
W(\vecv):= \{ \vecx \in P(\vecv)| (\vecv, \vecx, \vecw) 
\hbox{ is in the standard orientation} \}, 
\]
which is a component of $P(\vecv) - \bR \vecv$ and is independent of the choice of $\vecw$.
(But it is dependent on the choice of orientation of $\Lspace$.) 

A {\em crooked plane} in $\Lspace$ determined for a space-like vector $\vecu$ 
is given by a union of
\begin{itemize} 
\item a {\em stem} $p + \{\vecx \in \vecu^\perp: \vecx \cdot \vecx \leq 0\}$
which is a union of two closed cones of dimension $2$, 
\item and a pair of wings: 
\[p+  W(\vecu^-), p +   W(\vecu^+).\] 
\end{itemize} 
This is denoted by $\mathcal{C}(p, \vecu)$. 
Here, $p$ is called the origin of the crooked plane. 
(These are ``positively oriented ones". Of course, note that the definition depends on the orientation of $\Lspace$.)
Notice that the crooked planes are properly embedded disks topologically. 

%
%
%
%
A {\em crooked plane}  in $\Lspace/\Gamma$ is an injective image of 
a crooked plane. This is a $2$-cell which is not totally geodsics at two null lines which 
form the boundary of the stem.  Notice that 
we do not require proper embedding properties of 
the crooked plane in the image here. 

Let $\bR^4$ denote the standard Euclidean $4$-space with coordinates $x, y, z, t$.
The space of rays in $\bR^4$ identifies with a $3$-sphere $\SI^3$.
$\Lspace$ can be identifed as the subspace given by letting $t = 1$, and 
an open hemisphere in $\SI^3$ given by $t > 0$. 
Let $\mathcal{H}$ is the closed hemisphere that is the closure of $\Lspace$. 
We can consider $\Lspace$ as the affine subspace compactified by a $2$-sphere $\Ss$. 
We denote by $\Ss_+$ the space of rays in the positive cone. This forms a Klein model of the hyperbolic plane. 
There is the antipodal image $\Ss_-$, and we define $\Ss_0:= \Ss \setminus \Ss_+ \setminus \Ss_-$.
(See \cite{CDG22} for detail.)

Let $\mathcal{CR}(M)$ denote the space of  
immersed crooked planes with the quotient topology from $\mathcal{CR}(\Lspace)$ in 
Appendix \ref{app:metricspace}. 
Let $\mathcal{CR}(\Lspace)_K$ denote the the subset of $\mathcal{CR}(\Lspace)$ 
of crooked planes meeting $K$. 

Note that any isometry of $\Lspace$ is an affine transformation of $\Lspace$ which extends to a projective automorphism of $\SI^3$ acting on $\mathcal{H}= \clo(\Lspace)$. 
Hence, we have $\Isom(\Lspace) \subset \PGL(4, \bR)$ in a natural way.

The closure of a  crooked plane meets $\Ss_+$ in a complete geodesic. 
The orientation of the geodesic is not determined since 
$\mathcal{C}(p, \vecu) = \mathcal{C}(p, -\vecu)$. 
Conversely, an unoriented complete geodesic in $\Ss_+$ and the origin determines 
a crooked plane. 
Hence, the space $\mathcal{CR}(\Lspace)$ 
of crooked plane is in a one-to-one correspondence with 
$C(\Ss_+)\times \Lspace$ for 
the set $C(\Ss_+)$ of unoriented complete geodesics in $\Ss_+$.
(See Appendix \ref{app:metricspace} for the definition of this space.)

Let $\Gamma$ be a discrete isometry group acting on $\Lspace$ prorperly discontinuously and freely. 
We will assume that the linear part of $\Gamma$ preserves $\Ss_+$ in mostly. 
Let $\Sigma_+$ denote the complete hyperbolic surface $\Ss_+/\Gamma$.

\begin{itemize} 
\item A {\em horospherical cusp neighborhood} is 
an end neighborhood in $\Sigma_+$ corresponding to a parabolic fixed point in 
$\Ss_+$ which is an image of a horodisk where the parabolic element acts on. 
\item Let $c$ be a simple closed geodesic parallel to a boundary component of $\Sigma_+$. 
A {\em hypercyclic neighborhood} of $c$ in $\Sigma_+$ is a neighborhood of $c$ that is
the set of points of distance $< C$ for some constant $C> 0$. 
\item Now, $c$ bounds an end neighborhood corresponding the end,
which we call it the {\em geodesic end-neighborhood}. 
\item A {\em boundary neighborhood} is of an end of $\Sigma_+$ is either a horospherical cusp 
neighborhood or a hypercyclic neighborhood union the geodesic end-neighborhood. 
\end{itemize} 

If we take the union of the interiors of geodesic end neighborhoods, and remove it, 
we obtain a hyperbolic surface $\Sigma_+^*$ with geodesic boundary and some cusps. 

Recall that $\Gamma$ acts on $\Ss_+$ and the quotient is a complete hyperbolic surface
$\Sigma_+$. This is a finite-type complete hyperbolic surface with possible cusps and ends
that is compactified by a union of closed circles.

We will say that a crooked plane $P$ in $\Lspace$ is of {\em finite type} if $\clo(P) \cap \Sigma_+$ maps to a geodesic segment of finite length or it has a finite length outside a union
neighborhood of cusps or boundary parallel closed geodesics of $\Sigma_+$. 


A crooked plane $P$ in $\Lspace$ goes to an 
{\em immersed} crooked plane if each of the wings and stems of
a crooked plane $P$ goes to $\Lspace/\Gamma$ as an immersion. 
We do not require injectivity at all. 

We define a generalization of disjointness. 
We say that $\mathcal{C}(p_1, \vecu_1)$ and $\mathcal{C}(p_2, \vecu_2)$ in $\Lspace$ are 
{\em quasi-disjoint} if if they are the 
respective Chabauty  limits of sequences of pairwise disjoint crooked planes. 
(See Appendix \ref{app:metricspace}.) 

An embedded geodesic {\em spiral} along a simple closed geodesic $c$  if 
the distance to $c$ along the geodesic converges to $0$.  

A union of immersed crooked planes in $\Lspace/\Gamma$ is 
{\em quasi-injective} if the inverse image of it is a set of crooked planes in $\Lspace$ 
quasi-disjoint with respect to one another. 
The closure of a crooked plane $P$ in $\Lspace$ meets a geodesic arc $s(P)$ in $\Ss_+$. 
It is {\em cusp hanging in an end of $s(P)$} if the endpoint is a parabolic fixed point of $\Gamma$. 
It is {\em boundary spiraling in an end of $s(P)$} 
if the endpoint is a fixed point of a hyperbolic element of 
$\Gamma$.  These images in $\Lspace/\Gamma$ are called by the same names respectively.  
The points involved are 
called the endpoint of $P$ {\em cusp hanging} or {\em boundary spiraling} respectively. 
Each wing corresponds to the endpoint of $s(P)$. Hence, 
we call a wing is {\em cusp hanging} or {\em boundary spiraling} 
if the corresponding end of $s(P)$ is so
respectively. Also, the image in $\Lspace/\Gamma$ of the wing is called the same. 
(See Page \pageref{hanging} also.)

A crooked plane $P$ so that $s(P)$ gives us the shortest geodesics between 
boundary components of geodesic end neighborhoods or some horocyclic cusp neighborhoods of 
$\Sigma_+$ are called {\em short} crooked planes. 
These will not be boundary spiraling since they cannot give shortest geodesic between 
geodesic end neighborhoods. 

Danciger-Gu\'eritaud-Kassel  \cite{DGK162} showed that when there are no parabolics, 
$M$ can be decomposed by short crooked planes $P_1, \dots, P_m$. 
(See Definition 1.5 and Section 7.4 of their paper where we can infer that they are using these.) 

We can generalize this; however, we lose some embeddedness properties, particularly for parabolic holonomy.  These are in a generalized sense a decomposition since we can topologically change these by  isotopies to a cell decomposition, 

\begin{corollary} \label{cor:main}
Any Margulis space-time contains a quasi-disjoint and 
quasi-injective union of immersed short crooked planes so that the closure of 
each component of the complement is contractible
where the images of crooked planes may meet in finitely many points, which may include self-intersections at isolated points 
and the images of wings may coincide provided the wings are cusp-hanging ones. 
There are the only intersections. 
\end{corollary}

This answers  Conjecture of Charette-Drumm-Goldman partially. 
We need to improve to the embeddings and disjointness of the crooked planes to obtain 
the full conjecture. 
When $\Gamma$ has no parabolics, this was proved by Danciger, Gu\'eritaud, and Kassel \cite{DGK162}. 
Also, they announced the proof of the full conjecture some time ago  \cite{DGKp}. 


We give some outline of the proof: 
In Section \ref{sec:preliminary}, we review Chabauty topology and 
the parabolic actions on $\Lspace$.
In Section \ref{sec:deforming}, we will discuss about deforming the Margulis space-times. 
For deformation of Margulis space-times with parabolics, we use the earlier result from \cite{CDG22} and 
do the relative compactification to obtain $M_c$.  
This means adding a real projective surface $\Sigma$ 
that is a union of a two isometric 
complete hyperbolic surface $\Sigma_+$ and $\Sigma_-$ with cusps and some annuli. 
We remove the solid tori corresponding to parabolics to obtain a compact $3$-manifold $M_{c, r}$.
The boundary of each solid tori is an annulus ending in the boundary of 
a cusp neighborhood in $\Sigma_+$ and the corresponding one in $\Sigma_-$. 
See Section \ref{sub:before} for this. 

Then we deform $\Gamma$ to $\Gamma$':
We deform the solid tori first 
in Section \ref{sub:deformtori}. 
We deform $M_{c, r}$ to a compact $3$-manifold $M'_{c, r}$ with corresponding holonomy
using the theory of deformation of geometric structures by Thurston.
We glue back the solid tori deformed accordingly to obtain
a compact $3$-manifold $M'_{c, r, f}$ with a real projective structure
in Section \ref{sub:deform}. 


Now, we use the Carri\`ere's theory in Section \ref{sub:Carriere}.
The universal cover of $M'_{c, r, f}$ has a real projective structure. 
That is we show that the interior of $M''_{c, r, f}$ must be convex. 
For interior, the argument is not much different since we still have compact fundamental domain. 
Then the limit set consideration shows that the interior is 
a Margulis space-time without parabolics. 
We complete the proof of Theorem \ref{thm:main} in Section \ref{sub:generalcase}.

Now, we will use the theory of Danciger, Gu\'eritaud, and Kassel \cite{DGK162} for Margulis space-times without parabolics. 
These have disjoint crooked planes whose complement is a cell. 
These will be discussed in Section \ref{sec:crookedplanes}.  
First, we will discuss the disjointness of Crooked planes according to Drumm and Goldman in 
Section \ref{sub:disj}. We will also discuss about the homotopy classes of crooked planes in Section \ref{sub:genconj}.

Importantly, the homotopy classes of the crooked plane will change only finitely many times. 
From this we obtain the convergences to crooked planes that are disjoint but in some generalized sense.

%
%

In Appendix \ref{app:metricspace}, we will discuss about how a sequence of crooked planes can converge.

\subsection{Acknowledgement} 
We thank Virginie Charette, Jeffrey Danciger, Todd Drumm, and William Goldman
for many discussions.

\section{Preliminary} \label{sec:preliminary} 

Let $\mathcal{A}: \SI^3 \ra \SI^3$ sending $x$ to $-x$. 
For any set or a point $A$, we denote by $A_-$ the image $\mathcal{A}(A)$. 


For $x \in \partial \Ss_+$, we define $\zeta(x)$ to be the closed 
geodesic segment in $\Ss_0$ from $x$ to $x_-$ tangent to 
$\partial \Ss_+$ in the oriented direction.  
For any arc $\alpha$ in $\partial \Ss_+$, we define $Z(\alpha)$ the union 
of $\zeta(x)$ for $x\in \alpha$. This is a disk with boundary $\alpha$ and $\alpha_-$.

\subsection{Chabauty topology} \label{sub:Chabauty} 
We will be using the Chabauty topology for the subsets of $\Lspace$ and $\clo(\Lspace)$.

Let $C(\Lspace)$ denote the family of all closed subsets of $\Lspace$. 
For a compact $K \subset \Lspace$ and an open set $U$ of $\Lspace$, we set 
\begin{equation} 
\mathcal{U}_1(K):= \{C \in C(\Lspace)| C \cap K = \emp\}, 
\mathcal{U}_2(K):=\{C \in C(\Lspace)| C \cap U \ne \emp\}. 
\end{equation} 
The Chaubaty topology on $C(\Lspace)$ is the topology for which 
\[\{ \mathcal{U}_1(K)| K \hbox{ compact }\} \cup 
\{\mathcal{U}_2(K)| U \hbox{ open } \}\]
forms a subbasis. 

We recall that $C(\Lspace)$ is compact, and 
$C(\Lspace)$ is metrizable with a countable basis. 
(See Lemma E.1.1 of Benedetti-Petronio \cite{BP92}.) 

For example, 
note that $\{C_n\}$ in $C(\Lspace)$ converges to $\emptyset$
if for each compact subset $K'\subset \Lspace$, $C_n \cap K' =\emp$ for sufficiently large $n$. 
Also, in this case, $C_n$ is said to be {\em exiting}. 

Also, 
a sequence $\{C_n\}$ in $C(\Lspace)$ converges to $C \in C(\Lspace)$ if and only if 
the following two conditions are satisfied: 
\begin{itemize}
\item If $x \in \Lspace$ is such that there is a subsequence $\{C_{n_i}\}$ and $x_i \in C_{n_i}$ 
with $x_i \ra x$, then $x \in C$. 
\item Given $x \in C$, there exists $x_n \in C_n$ such that $x_n \ra x$ in $\Lspace$. 
\end{itemize} 
Here, we are using the standard Euclidean topology on $\Lspace$. 
(See Lemma  E.1.2 of \cite{BP92}.)

\begin{lemma} \label{lem:convS}
Under the Chabauty topology of $\Lspace$, the limit of a sequence of subspaces of dimension $i$ is again a subspace of dimension $i$ or is an empty set. 
A sequence of convex subsets of a subspace of dimension $i$ can converge only to
 a convex subset of dimension $\leq i$. 
\end{lemma} 
\begin{proof} 
The proof follows from the above item, i.e., Lemma E.1.2 of \cite{BP92}.
\end{proof} 

Obviously, 
a sequence of a convex segments always converges to a convex segment or the empty set. 

\begin{lemma} \label{lem:conv} 
Let $K_i$ be a sequence of closed convex domains in a subspace $P_i$ in $\Lspace$ 
with nonempty interiors. 
Suppose that $K_i$ is bounded by finitely many geodesics whose numbers are bounded uniformly. 
Let $K_i \ra K_\infty$ for a closed subset of $\Lspace$ under the Chabauty topology. 
Suppose that $P_i \ra P_\infty$ for a subspace $P_\infty$, and each boundary segments are convergent also. 
Then $K_\infty$ is a convex domain bounded by the union of the limit of the sequence of 
boundary geodesics of $K_i$
if $K_\infty$ has a nonempty interior in $P_\infty$. 
Otherwise, $K_\infty$ is a union of limits of the boundary geodesic segments.  
\end{lemma} 
\begin{proof} 
We can consider these as given by a linear equation for $P_i$  and finitely many inequatlities for 
sides of $K_i$ in coordinates. 
The zero sets of the inequalities could be assumed to be perpendicular to that of $P_i$. 
Then the argument follows by taking limits and Lemma E.1.2 of \cite{BP92}.  
\end{proof}

\subsection{Actions of parabolic groups} \label{sub:action}

\begin{definition} 
	Let $N$ be a nilpotent skew adjoint endomorphism. 
	We will call the frame $\va, \vb, \vc$ satisfying the following properties: 
	\begin{itemize} 
		\item $\vb = N(\va), \vc = N(\vb)$.
		\item $\va, \vc$ are null, and $\vb$ is of unit space-like.  
		\item $\Bs(\va, \vb)=0=\Bs(\vb, \vc), \Bs(\va, \vc) = -1$. 
	\end{itemize}
	an {\em adopted frame} of $N$.  
	We will say that $N$ is {\em accordant} if the adopted frame has the standard orientation. 
\end{definition}

\begin{corollary}[Corollary 3.8 in \cite{CDG22}] \label{cor:parac}
	Let $N$ be a nilpotent skew adjoint endomorphism. Then
the Lorentzian vectors $\va, \vb, \vc$ 
 satisfying the property that 
\begin{itemize} 
\item  $\Bs(\va, \vb)=0=\Bs(\vb, \vc), \Bs(\va, \vc) = -1$, 
\item $ \vc = N(\vb), \vb = N(\va)$, and 
\item $\vb$ is a unit space-like vector, $\vc \in \Ker N$ is {causally null},  and $\va$ is null
\end{itemize}
are determined up to changes 
$\vb \ra \vb + c_0 \vc, \va \ra \va + c_0 \vb + \frac{c_0^2}{2} \vc$ with respect to the  
a  skew-symmetric nilpotent endomorphism $N$ \red{and} $\Bs: V \times V \ra \bR$.  
Furthermore, the adopted frame for $N$ is determined only up to these changes and translations. 
\end{corollary} 

Under an adopted frame, 
the bilinear form $\Bs$ takes the matrix form
\begin{equation}\label{eqn:Bform}  
	\left(
	\begin{array}{ccc}
		0 & 0 & -1 \\
		0 & 1 & 0 \\
		-1 & 0 & 0
	\end{array}
	\right).
\end{equation}

Let $\gamma$ be a parabolic transformation $\Lspace \ra \Lspace$.
{Then it must be of the form 
	\begin{equation}\label{eqn:Phit0}
		\Phi(t) := \exp t
		\left(
		\begin{array}{cc}
			N & \vec{v}_\gamma \\
			0 & 0 
		\end{array}
		\right) \hbox{ for an accordant nilopotent skew adjoint element } N. 
\end{equation} }
In  the adopted frame with a choice of origin, we have the matrices
\begin{equation}\label{eqn:Phit}
	\Phi(t) := \exp t
	\left(
	\begin{array}{cccc}
		0  & 1  & 0  &0\\
		0  & 0  & 1  & 0\\
		0  & 0  &  0 & \mu \\
		0 &  0 & 0 & 0 
	\end{array}
	\right)
	= 
	\left(
	\begin{array}{cccc}
		1 & t  & t^{2}/2 & \mu t^{3}/6  \\
		0   & 1  & t  & \mu t^{2}/2  \\
		0  &  0  & 1 & \mu t \\ 
		0 & 0 & 0 & 1   
	\end{array}
	\right)
\end{equation} 
for $\mu \in \bR$. 
This was proved in Section 3.1 of \cite{CDG22}. 

\subsection{Invariant subspaces} 

A {\em strongly null half-plane} is a half-plane in a null plane bounded by a null line. 
A {\em hyperbolic transformation} is an element who linear part is diagonalizable
with some eigenvalues not equal to $\pm 1$.
Recall that a  hyperbolic element which acts without fixed points acts on a complete affine line as a translation. This line is called an {\em axis}.

\begin{lemma}\label{lem:noplane} 
	Let $\Gamma$ be the Margulis group with or without parabolics. 
	Then the following holds\/{\em :}
\begin{itemize} 
\item for hyperbolic $g\in \Gamma -\{\Idd\}$, acting without a fixed point, 
	$g$ acts on two null-planes each of which contains its axis and parallel to a vector 
	corresponding to the attractor or a repellor, and 
	there is no strongly null half-plane where $g$ acts on. 
	There is no other null-plane where $g$ acts on. 
\item	For parabolic $g \in \Gamma -\{\Idd\}$, acting without a fixed point, there is no null plane or strongly null half-plane that it
	acts on.
\end{itemize}  
\end{lemma} 
\begin{proof}
	Let $g$ be hyperbolic.  Then there are two independent null eigenvectors and a spacelike eigenvector. 
	The linear part $\mathcal{L}(g)$ acts on two null subspaces spanned by 
	a common  space-like eigenvector and two other null eigenvectors respectively. 
	There is no other null vector space where $\mathcal{L}(g)$ acts on
	since there is a nondegenerate bilinear form and the duality. 

 In the direction of 
space-like eigenvector, 
we can show that $g$ acts on a unique affine line $L$
by linear algebra.  And $g$ acts as a translation on $L$. 
There are two null-planes each of which contains $L$ and any point from a point of $L$ moving 
by one of the two null-eigenvectors. 
	
	Let $P$ be a null affine plane where $g$ acts on.  The closure of $P$ in $\SI^3$ 
contains 	four fixed points corresponding to the eigenvectors. 
	Since the linear part of $g$ 
	has a contracting direction, there must be an invariant line where $g$ acts as
a translation, which is $L$. 
	Hence, there are only two invariant affine null planes. 
	
	Strongly null half-space must be in one of the two. The action of $g$ in the affine planes shows that 
	there cannot be such null half-spaces. 
	
	Let $g$ be parabolic.  Then the linear part $\mathcal{L}(g)$ acts on the null space 
	tangent to the null vector which is the unique eigenvector. 
	Since a null vector plane contains a unique point of $\partial \Ss_+$, and 
	$\mathcal{L}(g)$ fixes a unique point of $\partial \Ss_+$, it follows that 
	there is unique null vector plane where $\mathcal{L}(g)$ acts on. 
We use the adopted frame of $\mathcal{L}(g)$. 
	The vector subspace is spanned by null $(1, 0, 0)$ and space-like $(0, 1, 0)$. 

It is clear that $\mathcal{L}(g)$ cannot act on any other vector space. 
	
	Suppose that there is an affine plane where $g$ acts on. 
	Then its direction is spanned by $(1, 0, 0)$ and $(0, 1, 0)$ by the above paragraph. 
	From the expression \eqref{eqn:Phit} of $g$, 
	we see that such an affine plane must move by $\mu t$ in the third coordinate. 
	This is a contradiction. 

Since a strongly null affine halfplane is in a null affine plane, we cannot have an invariant one as well.
\end{proof}

\section{Deforming} \label{sec:deforming} 

We will prove Theorem \ref{thm:main} in this section. 
Let a Margulis group $\Gamma \subset \Isom(\Lspace)$ be generated by $g_1, \dots, g_n$. 
Since $\Gamma$ is free, we may suppose that they are free generators. 

We will identify $\pi_1(M)$ with $\Gamma$ and the free group $F_n$ of rank $n$ here. 
Note that $\Hom(F_n, \Isom(\Lspace))$ can be identified with 
$\Isom(\Lspace)^n$ by sending 
\[\rho \mapsto (\rho(g_1), \rho(g_2), \dots, \rho(g_n)).\] 
Since $\Hom(F_n, \Isom(\Lspace)) = \Isom(\Lspace)^n$ is a real algebraic set. 
We use the product topology which is locally path-connected Riemannian geodesic-metric space.  
For $\mu$ in $\Isom(\Lspace)^n$,  
we define $\Gamma_\mu$ to be an image of holonomy 
homomorphism $h_\mu: \pi_1(M) \ra \Isom(\Lspace)$
where $(g_1, \dots, g_n) \mapsto \mu$. 

Let $\Pi: \SO(2, 1) \ra \PSO(2, 1)=\SO(2, 1)/\{\pm \Idd\}$ denote the quotient homomorphism.
A representation $h: F_n \ra \Isom(\Lspace)$ is {\em Fuchian} if 
$\Pi \circ \mathcal{L}\circ h:\pi_1(S)$ is a Fuchsian representation.


We note that the set of discrete faithful representations 
\[\mathcal{F}_n\subset \Hom(F_n, \PSO(2, 1))/\PSO(2, 1)\] 
is a union of open cell with some boundary points by Fricke.  
We call $\mathcal{F}_n$ the {\em Fricke subspace}. 
This is a semi-algebraic cell with boundary where the boundary consists of 
representations with parabolics. 
Here only the boundary parallel homotopy classes can be parabolic by 
the theory of Fricke. 

We recall the standard cone neighborhood theory following from 
the triangulations of the semialgebraic sets as in Theorem 9.2.1 of  \cite{Bochnak}: 
We define a {\em cone} to be a semi-algebraic space semi-algebraically 
homeomorphic  to a product $X\times [0, 1]$ collapsing $X\times \{0\}$ to a point. 
A point $x$ of a semi-algebraic set has a neighborhood with a semi-algebraic homeomorphism to a cone over a semi-algebraic set S in the boundary of the neighborhood with a cone-point at the origin corresponding to $x$. 

\begin{lemma} \label{lem:deformnop} 
Let $\rho:F_n \ra \Isom(\Lspace)$ be an affine deformation of a Fuchian representation
with $\mathcal{L}\circ \rho$ is a free group with a parabolic element. 
There is a cone-neighborhood $N$ of $\rho$ where 
the subset $N'$ of $N$ of representations without parabolics is 
an open dense cone which is a complement of finitely many hyperspaces in $N$ passing $\rho$, which are also cones. 
\end{lemma}
\begin{proof} 
Let $D_n \subset \Hom(F_n, \PSO(2,1))$ denote the inverse image of $\mathcal{F}_n$. 
The subset $\hat D_n$ of $D_n$ where some parabolic elements exist is a union of semi-algebraic subsets of codimension $\geq 1$ passing $\rho$
since only elements parallel to boundary can become parabolic. 
The subset in question is the inverse image of $\hat D_n$. 

There is a map 
$\mathcal{L}': \Hom(F_n, \Isom(\Lspace)) \ra \Hom(F_n, \PSO(2, 1))$ by sending 
$h(g)$ to $\Pi\circ \mathcal{L}(h(g))$ for $g \in F_n$. 
This can be considered as a fiber bundle
$(\Pi\circ\mathcal{L})^n:\Isom(\Lspace)^n \ra \PSO(2, 1)^n$
with the fibers isomorphic to $(\bZ_2\times \bR^{2, 1})^n$. 
The proposition follows from the standard 
cone-neighborhood theory of semi-algebraic set. 
%
\end{proof} 



\subsection{Before deformation} \label{sub:before} 
To begin, we suppose that $\mathcal{L}(\Gamma) \subset \Isom^+(\Lspace)$. 

Let $\rho$ be a proper affine deformation of a Fuchian representation with parabolics. 
Let $\Gamma$ denote its image. 
We recall some facts from \cite{CDG22}. 
Let $\Sigma_+ := \Ss_+/\Gamma$ and $\Sigma_- := \Ss_-/\Gamma$ are real projective surfaces. 
$\Gamma$ acts as Fuchsian group acting on the hyperbolic plane $\Ss_+$, and 
acts properly discontinuously and cocompactly 
on $\Ss_+ \cup a$ for a union $a$ of arcs in $\partial \Ss_+$. 
Then $\Sigma_+$ is dense in a complete surface $\hat \Sigma_+:= \Ss_+\cup a/\Gamma$ with cusps corresponding 
to the parabolic elements in $\Gamma$. 
The union $a$ covers a finitely many boundary components $a_1, \dots, a_m$ of $\hat \Sigma_+$. 
We choose a single lift $\tilde a_i$ for each $i$. 
$a$ is a union of images of $\tilde a_i, i=1, \dots, m$ under $\Gamma$. 
There are also mutually disjoint cusp neighborhoods $E_i$, $i=1, \dots, k$. 
For each $i$, we choose a horoball $\tilde E_i$ covering $E_i$, $i=1, \dots, k$. 
Also, $E_i$, $i=1,\dots, k$, form cusp neighborhoods in $\Sigma_+$, and there are corresponding 
cusp neighborhoods $E_{i, -}$ in $\Sigma_-$ for $i=1, \dots, k$. 
We denote by $E$ the union of $E_i, i=1, \dots, k$ and $E_-$ the union of $E_{i,-}$ for $i=1, \dots, k$. 

Let $\partial \Ss_+$ denote the boundary of $\Ss_+$ with the boundary orientation.
Then there is a domain $D_i= Z(\tilde a_i)$ in $\Ss_0$ for $i=1, \dots, m$ that is a union of 
maximal geodesic in $\Ss_0$ tangent to $\partial \Ss_+$ in the boundary direction. 
We also have 
$\partial D_i = \tilde a_i \cup \tilde a_{i, -}$. 
Then $\Gamma$ acts properly discontinuously on the union of images of $D_i$ under 
$\Gamma$. On each $\tilde a_i$, there is a unique primitive element $\eta_i$
acting on it in the orientation direction. It also acts on $D_i$. 

We define $\tilde \Sigma := \Ss_+ \cup \Ss_- \cup \bigcup_{i=1}^m \Gamma(D_i)$
where $\Gamma$ acts properly discontinuousy, 
and define $\Sigma := \tilde \Sigma/\Gamma$.  
(See Section 5.1 of \cite{CDG22} and Theorem 5.3 of \cite{CG17}.)
Then $\Sigma$ contains $\Sigma_+$ and $\Sigma_-$ and the complement is 
a union $A$ of annuli $A_i$ which is an image of $D_i$, $i=1, \dots, m$. 

Here, $\Sigma$ is a union of 
\[\Sigma_+, \Sigma_-, \hbox{ and } A_i, i=1, \dots, m.\] 
By \cite{CDG22},  $\Gamma$ acts properly discontinuously and freely on
$\Lspace \cup \tilde \Sigma$ whose quotient is a union of $M$ and a surface $\Sigma$. 
Let $p$ denote the covering map. (See Section 5.1 of \cite{CDG22}.) 



By Section 5.1.1 of \cite{CDG22}, we have a mutually disjoint collection of solid tori 
$T_1, \dots, T_m$ closed in $\hat M_0:= (\Lspace/\Gamma) \cup \Sigma$, 
and a compact $3$-manifold $M_1:= \hat M_0 - \bigcup_{i=1}^m T_i^o$ is a deformation retract of $M_0$ where 
$M_1 \cap T_i$ is an annulus ${\mathfrak{A}}_i, i=1,\dots, k,$ with boundary components 
$\alpha_i^+$ and $\alpha_i^-$ in $\Sigma_+$ and $\Sigma_-$ respectively. 
These have parabolic holonomies. 
Here, $\Sigma_e = \Ss_e/\Gamma$ for $e=+, -$ is a complete hyperbolic surface 
where each cusp has the cusp neighborhood is bounded by $\alpha_i^e$. 
Also, ${\mathfrak{A}}_i$  is a surface ruled by time-like geodesics. (See Section 3.3 and 
Appendix A of \cite{CDG22}.)

Let $\gamma_i$ denote the homotopy class corresponding to $\alpha_i$ acting on $\tilde \alpha_i$ for each $i$. 
In $\Lspace$, ${\mathfrak{A}}_i$ is covered by disks $\tilde {\mathfrak{A}}_i$ with boundary components 
equal to a horocycle in $\Ss_+$ and its antipodal image in $\Ss_-$,
where $\tilde {\mathfrak{A}}_i$ is $\Phi_{i, s}$-invariant for a one-parameter group $\Phi_{i,s}$, $s\in \bR$, of parabolic isometries containing $h(\gamma_i)$. 

Also, $\tilde {\mathfrak{A}}_i$ bounds a $3$-cell $\tilde T_i$ covering $T_i$. 

The inverse image $p^{-1}({\mathfrak{A}}_i)$ is the union of 
a disjoint copy of $\tilde {\mathfrak{A}}_i$ under $\Gamma$, 
and $\partial M_1$ is covered by $(\tilde \Sigma  \setminus E) \cup  
\bigcup_{i=1}^m p^{-1}({\mathfrak{A}}_i)$.

\subsection{Deforming solid tori}\label{sub:deformtori} 
Now, we will try to make these noncompact annulus times a real line 
into compact ones by deforming their boundary
components. 
Now assume $h_\mu$ be in a cone neighborhood $N$ of $h:=\rho$ in $\mathcal{F}_n$
for $\mu \in \Isom(\Lspace)^n$ 
described in Lemma \ref{lem:deformnop}
without parabolic elements. We let $N'$ denote the complement $N$ of 
set of representations with parabolics. 
Here, $\mu$  denote the parameter of $N$ and we assume $\mu \in \bR^{6n}$ and $\mu_0$ is the origin and $N$ corresponds to a ball. 
(Note $\dim \Isom(\Lspace) = 6$.) 
Let $\mu_0$ correspond to $h=\rho$. 
We will choose a sufficiently small neighborhood in $N$ so that the completeness can be 
established. 

We now describe how these solid tori deforms and we compactify these as well
while not concerned about $M_1$ for a while: 

As we deform $\Gamma$, we have changed
$h(\gamma_i)$ to hyperbolic element $h_\mu(\gamma_i)$ for each $i=1, \dots, k$, 
$x\in N$.  We denote by $\Gamma_\mu$ the image of $h_\mu$ in $N$. 


We first deform $T_i$ bounded by ${\mathfrak{A}}_i$ where
$T_i$ is diffeomrphic to ${\mathfrak{A}}_i\times \bR$. 

Define 
$\hat T_i :=(\Lspace \cup \Ss_+\cup \Ss_-)/\langle \rho(\gamma_i)\rangle$ where $T_i$ embeds
as a submanifold  bounded by $\hat {\mathfrak{A}}_i$ that is the image of ${\mathfrak{A}}_i$. 

Here, we need to choose $N$ to be smaller if necessary for the argument to work. 

We define $\hat T_i(\mu):= (\Lspace \cup \Ss_+\cup \Ss_-)/\langle h_\mu(\gamma_i)\rangle$. 
We can choose a compact neighborhoods $N_i$ of ${\mathfrak{A}}_i$ in $\hat T_i$ 
and a compact neighborhood $N_i(\mu)$ of ${\mathfrak{A}}_i(\mu)$ in $\hat T_i(\mu)$. 
There is a diffeomorphism $f_i(\mu): N_i \ra N_i(\mu)$ 
lifting to a diffeomorphism $\tilde f_i(\mu): \tilde N_i \ra \tilde N_i(\mu)$. 
We may assume that $\tilde f_i(\mu) \ra \Idd$ in the $C^r$-topology for $r \geq 2$: 
We construct $\tilde f_i(\mu)$ as a parameter of immersions 
as in  the proof of Thurston and Morgan 
written-up by Lok \cite{Lok84} by considering covering 
$(\Lspace \cup \Ss_+\cup \Ss_-)$
by precompact balls and defining smooth isotopies 
starting from the compact sets in the intersections of maximal number of balls. 
We can choose $\tilde f_i(\mu)$ so that there is $C^{r}, r \geq 2$, convergence as 
$\mu \ra \mu_0$. 
Since $N_i$ is compact, it is covered by finitely many precompact balls $B_{i, k}, k=1, \dots, l_i$ for some $l_i$ where  $f_i(\mu)$ is an embedding:
Since we are working with a regular neighborhood of a compact annulus, 
we may use only two precompact balls intersecting at balls. 
We can use Theorem 1.5.3 and Remark 1.5.5 in \cite{CEG06} by identying $N(\mu)$
to show that 
the induced map $f_i(\mu): N_i \ra N_i(\mu)$ is an embedding 
for $\mu$ sufficiently close to  $\mu_0$ where we may need to choose smaller $N_i$. 
Hence, we will choose $N_i(\mu)$ as the image of $f_i(\mu)$ for sufficiently small $t$. 

Since $\tilde f_i(\mu)$ lifts $f_i(\mu)$, we have 
$h_\mu(\gamma_i) \circ \tilde f_i(\mu) = \tilde f_i(\mu) \circ h(\gamma_i)$. 
We define $\tilde {\mathfrak{A}}_i(\mu)$ as the image $\tilde f_i(\mu)(\tilde {\mathfrak{A}}_i)$. 
Since $\tilde {\mathfrak{A}}_i(\mu)$ covers a compact annulus ${\mathfrak{A}}_i(\mu)$ in $\hat T_i(\mu)$, 
$\tilde {\mathfrak{A}}_i(\mu)$ is properly embedded in $\Lspace \cup \Ss_+ \cup \Ss_-$
as we can see from the action of $h_\mu(\gamma_i)$. 
Now, $\langle h_\mu(\gamma_i) \rangle$ acts freely and properly discontinuously 
on $\tilde {\mathfrak{A}}_i(\mu)$. 

The arc $\Ss_e \cap \tilde {\mathfrak{A}}_i(\mu)$, $e=+, -$ is an arc $\alpha_i^e(\mu)$ that bound a convex disk 
$\tilde E_i^e(\mu) \subset \Ss_e$ for $i=1, \dots, k$. 
The boundary $\tilde E_i^e(\mu)$ is a disjoint union of $\alpha_i^e(\mu)$ and 
an arc $\beta_i^e(\mu)$ in $\partial \Ss_e$.

We construct the region $\tilde F_j(\mu):= Z(\beta_j^{e, o}(\mu))$ in $\Ss_0$.
We define 
\[\tilde \Sigma_\mu:= \Ss_+ \cup \Ss_- \cup \bigcup_i \Gamma_\mu Z(a_i(\mu)) \cup 
\bigcup_i \Gamma_\mu\tilde F_i(\mu) \subset \SI^3.\]
$\Gamma_\mu$ acts properly discontinuously on it by Theorem 5.3 of \cite{CG17}. 
We define $\Sigma_\mu:= \tilde \Sigma_\mu/\Gamma_\mu$. 
Here, we need $|\mu| < \eps_0$ for some constant $\eps_0 >0$ for this argument to work. 



We do not yet know that $\Gamma_\mu$ acts properly and freely on $\Lspace$, 
and we need to construct some compact $3$-manifold. 


By Theorem  5.3 of \cite{CG17}, 
the  closed surface $\Sigma_\mu$ is a union of 
\begin{itemize} 
\item two complete hyperbolic 
surfaces $\Sigma_{+, \mu}:= \Ss_+/\Gamma_\mu$ and
\item  $\Sigma_{-, \mu}:= \Ss_-/\Gamma_\mu$ and 
\item Annuli $A_{i, \mu}$, 
images of $Z(a_i(\mu))$, $i=1, \dots, m$. 
\item The annular images  of strips 
$\tilde F_i(\mu)$  which we denote by $F_{i, \mu}$, $i=1, \dots, k$.  (``opened-up ones''.) 
\end{itemize} 
Here $\partial \tilde D_i(\mu) = a_i(\mu) \cup a_{i}(\mu)_-$ which covers the union of 
a boundary component of 
the closure of $\Sigma_{+, \mu}$ and that of $\Sigma_{-, \mu}$. 
$\partial \tilde F_i(\mu)$ also covers the boundary components of these, 
which were opened up from cusps.  


Then the closure of $\tilde {\mathfrak{A}}_i(\mu)$ in $\Lspace \cup \Sigma_\mu$ bound
a $3$-cell $E_i(\mu)$ so that $T_i(\mu):= E_i(\mu)/\langle h_\mu(\gamma_i)\rangle$ is a compact solid torus. 
The boundary of $T_i(\mu)$  is a union of four annuli ${\mathfrak{A}}_i(\mu):= \tilde {\mathfrak{A}}_i(\mu)/\langle h_\mu(\gamma_i)\rangle $, 
$E_i^e:= \tilde E_i^e/\langle h_\mu(\gamma_i)\rangle$, 
and $D_i(\mu):= \tilde D_i(\mu)/\langle h_\mu(\gamma_i)\rangle$. 


\subsection{Deforming the compact part} \label{sub:deform}

Now, we deform $M_1$. We do this by deforming ${\mathfrak{A}}_i$ by 
${\mathfrak{A}}_i(\mu)$. As we deform $\Gamma_\mu$, we obtain the deformation of 
the real projective structure on $M_1$ by first deforming the boundary ${\mathfrak{A}}_i$
to an annulus ${\mathfrak{A}}'_i(\mu)$ for each $i$. 

Again, we need to choose smaller $N$ if necessary for the following argument to work. 

We deform by $\tilde f_i(\mu)$ to $\tilde {\mathfrak{A}}_i$ but we will apply to the one-sided neighborhood 
$\tilde N_i$ of $\tilde {\mathfrak{A}}_i$ inside $\tilde M_1$ as follows: 


We can consider $\Isom(\Lspace) $ as a subgroup of $\PGL(4, \bR)$ acting on $\SI^3$ projectively. 
We construct a flat $\SI^3$-bundle over $M_1$ by the product
$\tilde M_1\times \SI^3$ and taking the quotient by the diagonal action of $\Gamma$.
It has an obvious flat connection by mapping leaves $\tilde M \times x$ for $x \in \SI^3$ for all $x$. 
The geometric structure based on $(\SI^3, \PGL(4, \bR))$ is given 
by a section transversal to the flat connection. 

At $\mu=0$, we have a transversal section. For sufficiently small $\mu$, we 
can still obtain the deformed $M_1$ with holonomy $h_\mu$ and 
the boundary deforming $\Sigma$ and ${\mathfrak{A}}_i$.  
We deform first for ${\mathfrak{A}}_1$ to become ${\mathfrak{A}}_1(\mu)$ while deforming  
$\Sigma - E^o$ inside $\Sigma_\mu$ only and we use 
the partition of unity to extend the transversal section everywhere on $M_1$. 
(See Goldman \cite{G88} for detail).  
Here, we need to take smaller $N$ if necessary. 


We call the result $M_1(\mu)$. Now, $M_1(\mu)$ may not necessarily be a quotient of a domain 
in $\clo(\Lspace)$ unlike $M_1$.  However, this will follow later.


Now, we glue $T_i(\mu)$ to $M_1(\mu)$ by identifying ${\mathfrak{A}}_i(\mu)$ and $A'_i(\mu)$ in an obvious manner
for each $i=1,\dots, m$. 
The result is a manifold $M_2(\mu)$.  It is compact since so is $M_1(\mu)$ and $T_i(\mu)$.
Also, we can verify that $\partial M_2(\mu) = \Sigma_\mu$.


\subsection{Completeness using the argument of Carri\`ere.} \label{sub:Carriere} 

We will show $M_2(\mu)^o$ is covered by $\Lspace$. 
To do that we will prove that $M_2(\mu)^o$ is convex. 

Since $M_2(\mu)$ is compact, $\tilde M_2(\mu)$ has a compact fundamental domain 
$F(\mu)$. 

For now, $\dev_\mu: \tilde M_2(\mu) \ra \clo(\Lspace)$ is a smooth immersion. 
Now, $\dev_\mu|\tilde M_2(\mu)^o \ra \Lspace$ is one also. 

Again, using the theory of \cite{psconv}, we obtain a Kuiper completion 
$(\tilde M_2(\mu)^o)\che{}$ where we complete the path-metric induced from the pulled-back Riemannian metric of the Euclidean metric of $\Lspace$. 
We denote by $\delta_\infty \tilde M_2(\mu)^o$ the set of ideal points. 
Also, the developing map extends to $(\tilde M_2(\mu)^o)\che{}$ which we denote by 
$\dev_\mu$. Also, $\dev_\mu(\delta_\infty \tilde M_2(\mu)^o)$ is in $\Lspace$. 

From here on, we fix $\mu$. 
Suppose  that $\tilde M_2(\mu)^o$ is not convex. 
Then there exists a triangle $T$ in $(\tilde M_2(\mu)^o)\che{}$ 
with vertices and two edges in $\tilde M_2(\mu)^o$
and the interior of
 one edge $E_v$ meeting $\delta_\infty \tilde M_2(\mu)^o$
with the opposite vertex $v$. 
The two other edges are disjoint from $\delta_\infty \tilde M_2(\mu)^o$. 
(See for example \cite{CG93} where we have a slightly different metric but the arguments 
are the same.)


Consider the open domain $D_v$ in $\tilde M_2(\mu)^o$,
 which can be reached by a segment from $v$. 
Then $\dev_\mu| D_v$ is an embedding to $\Lspace$. 
$T$ is in the closure $\clo(D_v)$. 
By a perturbation of $T$ inside $\clo(D_v)$ as in  Proposition 4.3 of \cite{psconv}, 
we can assume that there exists a triangle $T'$ which meets 
$\delta_\infty \tilde M_2(\mu)^o$. 
We may assume that this is a single point only at an interior point of an edge say $E'_v$
by perturbing the edge $E_v$ in $\clo(D_v)$. 
Now, we let $T$ be this $T'$ and $E_v$ be this $E'_v$ by renaming.

Let $l$ be a maximal segment in $E_v$ ending at this ideal point. 
Then $p_t: l\ra M_2(\mu)$ goes to an infinite length arc. 
Hence, there is a sequence $p_i\in l$ and $g_i \in \pi_1(M)$ so that 
$g_i(p_i) \in F(x)$ where $g_i$ form an unbounded sequence.

Let $F_2(\mu)$ denote the finite compact union of images of $F(\mu)$ adjacent to $F(\mu)$. 
There is a lower bound $\eps >0$ for any path from $F(\mu) \cap \tilde M_2(\mu)^o$ 
to $\tilde M_2^o(\mu) \setminus F_2(\mu)^o$ 
with respect to the induced Euclidean metric $d_E$ on $\tilde M_2(\mu)$: 
This follows since there is a lower bound on the induced spherical metric on 
$\tilde M_2(\mu)$ form $F(\mu)$ to $\tilde M_2 \setminus F_2(\mu)^o$ since one is compact 
and the other is a disjoint closed set on $\tilde M_2(\mu)$ with induced spherical metric.  
On $\Lspace$, the Euclidean metric strictly dominates the spherical metric. 

Now $g_i(p_i)$ is contained in a Euclidean ball  $E_i$ in $F_2(\mu) \cap \tilde M_2(\mu)^o$
as a center $g_i(p_i)$ and radius $\eps$. 

As Carri\`ere \cite{Carr89} showed $g_i^{-1}(E_i)$ is a sequence of ellipsoids with center $p_i$ 
so that the euclidean diameter $\diam_{d_E}(g_i^{-1}(E_i)) > \delta$ for $\delta> 0$. 
This follows since the discompactedness of the linear 
part in the Lorentz group $\SO(2,1)$ in $\Isom(\Lspace)$ is $1$. 
Hence, $g_i^{-1}(E_i)\cap T$ 
contains a segment $s_i$ with an end point $p_i$ of euclidean length at least 
$\delta> 0$. 
Hence, $g_i^{-1}(E_i) \cap T$ contains a point $x_i$ of $d_E(x_i, p_i) =\delta/2$ 
converging to $x \in T \setminus \delta_\infty M_2(\mu)^o$ up to a choice of subsequences 
 by the geometry of the situation.  Since $g_i(x_i)$ is in the compact Euclidean $\delta_2$-neighborhood of $F_i$, 
this is a contradiction to the proper discontinuity of the action of the fundamental group
on $\tilde M_2$. 

We remark that we still have compact fundamental domain and the argument still works
although the fundamental domain is not in $\Lspace$.


Therefore, $M_2(\mu)^o$ is convex and $\dev_\mu|\tilde M_2(\mu)^o$ is an embedding to 
a convex domain $\Omega_\mu$ in $\Lspace$. 
Also, $\Omega_\mu$ contains the domains $\tilde E_i(\mu)$. 
Hence $\clo(\Omega_\mu)$ contains the strips $\tilde {\mathfrak{A}}_i(\mu)$ and its $\Gamma_\mu$-images. 
The set of $\Gamma_\mu$-images of these points is contained in 
$\clo(\Omega_\mu)$. 
Since $\tilde {\mathfrak{A}}_i^o$ and $\tilde A_i^o$ are in $\tilde M_2(\mu)^o$, 
we see that the fixed points of 
$h_\mu(g_i)$ and $h_\mu(\gamma_i)$ in $\partial \Ss_+$ and its antipodes in 
$\partial \Ss_-$ are in the closure of $\Omega_\mu= \dev_\mu(\tilde M_2^o)$. 
Also, the fixed points of its conjugates are in it also. 
The closure of the set of these points is the limit set $\Lambda_\mu$.
From the properties of the Fuchian groups, the convex hull of $\Lambda_\mu$ 
is a convex domain $E_\mu$, and 
so is the convex hull of the limit points in $\partial \Ss_-$ which must equal 
$\mathcal{A}(E_\mu)$. The convex hull in $\clo(\Lspace)$ of
$E_\mu \cup \mathcal{A}(E_\mu)$ is $\clo(\Lspace)$ as we can take the interior point $x$ and its
antipodes and their neighborhoods in $E_\mu$ and $\mathcal{A}(E_\mu)$. 
Since $\clo(\Omega_t)$ is a convex domain in $\clo(\Lspace)$
containing $E_\mu$ and $\mathcal{A}(E_\mu)$, 
the only possibility is that $\clo(\Omega_t) = \clo(\Lspace)$. 
Hence, $\Omega_\mu =\Lspace$. 



\subsection{General case of $\Isom(\Lspace)$.} \label{sub:generalcase} 

We will now consider the Margulis space-times with parabolics where 
$\mathcal{L}(\Gamma)$ is not in $\SO(2, 1)^+$. 
Again, we have a representation $h: F_n = \pi_1(M) \ra \Isom(\Lspace)$ 
whose image is $\Gamma$. 
We can consider $F_n$ as $\langle G, g\rangle$ where 
$h(G) \subset \Isom^+(\Lspace)$ and $g$ is not in $G$ for a subgroup $G$ of index $2$ in $F_n$
acting on $\Ss_+$. 
Note there are deformations $h_\mu$ giving us 
$h_\mu|G$ and $h_\mu(g)$ for $t\in [0, 1]$. 
Now, $\Lspace/G$ covers $\Lspace/\Gamma$ by $2$-to-$1$ map with 
a deck transformation group generated by a map induced by $g$ 
isomorphic to $\bZ_2$. 
We denote by $G_\mu$ the group deformed from $G$.
Since $G$ is a free group of rank $\geq 2$, we can apply the results of the above sections
to $M_1:=\Lspace/G$ to obtain a deformation 
$M_1(\mu):=\Lspace/G_\mu$ without parabolics.  

Now, $h(g)$ induces a $\bZ_2$-group action on $M$ without fixed points. 
Let $\hat g$ denote the generator of this action.
We can choose the parabolic torus $T_1, \dots, T_k$
so that each one is sent to another one since $\hat g$ 
does not preserve a parabolic end of $\Sigma_+$
by Section 5.1 of \cite{CDG22}.  

Let  $T:=T_1\cup \dots \cup T_n$. 
We choose $T_i(\mu)$ as in Section \ref{sub:deformtori}. 
We let $T(\mu):= T_1(\mu) \cup \dots \cup T_n(\mu)$. 

Hence, $\hat g$ acts on a compact set $M_1$ freely as well. 
$h_\mu(g)$ still induces a $\bZ_2$-group action on 
a compact manifold $M_1(\mu) \setminus T(\mu)^o$ 
for some choices of $T_i(\mu)$ near $T_i$ deformed 
from ones in the above section. For example, we can choose half of these and let others be the images under the induced map of $h_\mu(g)$ since the action on $T_1,\dots, T_k$ 
does not preserve any components. 


If there is a fixed point of the $\bZ_2$-action on a compact manifold $M_1(\mu) \setminus T(\mu)^o$,
then there must be a fixed point of $h_\mu(g')$ on the universal cover of 
$M_1(\mu) \setminus T(\mu)^o$ contained in $\Ss_+ \cup \Ss_- \cup \Lspace$
for $g'' =g g'$ for $g'\in \pi_1(M)$.
This is a contradiction since $h_\mu(\pi_1(M))$ acts freely there. 

Since this action also sends each $T_i(\mu)$ to a different one, and 
the $\bZ_2$-action on $M_1(\mu)\setminus T(\mu)^o$ is free, it follows that 
the $\bZ_2$-action on $M_1(\mu)$ is free also. 
Hence, we obtained a deformation from $\Lspace/\Gamma$ to 
$\Lspace/\Gamma_\mu$ for sufficiently small $t$. 

This completes the proof of Theorem \ref{thm:main}. 

%

\section{Crooked plane decomposition} \label{sec:crookedplanes} 

Here, we directly work with representations with linear parts in
 $\SO(2, 1)$. 

\subsection{Disjointness of crooked planes} \label{sub:disj} 

A pair of space-like vectors $\vecu_1, \vecu_2 \in \bR^{2, 1}$
is said to be {\em crossing} (resp. {\em ultraparallel}, {\em asymptotic}) 
if $\vecu_1^\perp  \cap  \vecu_2^\perp$ is time-like (resp. space-like, null). 
Alternatively, 
\[
(\vecu_1 \cdot \vecu_2)^2 - (\vecu_1 \cdot \vecu_1)(\vecu_2 \cdot \vecu_2) < 0, 
\hbox{( resp. } >0, =0). 
\]
A {\em noncrossing pair} is one that is either ultraparallel or asymptotic. 
We say that $\vecu_1, \vecu_2 \in \bR^{2, 1}$ are {\em consistently oriented } 
if $\vecu_1 \cdot \vecu_2 < 0, \vecu_i \cdot \vecu_j^{\pm} \leq 0$ for $i, j = 1, 2$. 
Up to multiplication by $-\Idd$ to one or two of space-like non-crossing $\vecu_i, i=1, 2$, 
we can make $\vecu_1$ and $\vecu_2$ be consistently oriented.
(See \cite{CK14} for details.)

Also, note that $\mathcal{C}(p, \vecu) = \mathcal{C}(p, -\vecu)$. 

We define 
\[D(\vecu) := \{ a \vecu^- - b \vecu^+| a, b \geq 0\} -\{0\}.\]

Given a set $A \in \bR^3$, we denote by 
$-A$ the set $\{-a| a \in A\}$. 
In the following, $A - B$ for subsets $A, B \in \bR^3$ means the 
set $\{ x-y | x \in A, y\in B\}$, which is just a convex hull of 
the set $A$ and $-B$ if $A$ and $B$ are convex cones. 
\begin{theorem}[Drumm-Goldman \cite{DG99}]\label{thm:disj}  
Suppose that $\vecu_1, \vecu_2$ are nonparallel, 
consistently oriented, noncrossing space-like vectors.
Then the crooked planes $\mathcal{C}(p_1, \vecu_1)$ and 
$\mathcal{C}(p_2, \vecu_2)$ are disjoint if and only if 
$p_2 - p_1 \in \mathrm{int}(D(\vecu_2) - D(\vecu_1))$. 
\end{theorem}


Any pair of immersed crooked planes in $\Lspace/\Gamma$ that always lift to a quasi-disjoint pair of crooked planes are also said to be {\em quasi-disjoint}.
In particular, this is true if the two are respective Chabauty limits of sequences of crooked planes that are pairwise disjoint.


Note that a crooked plane separates $\Lspace$ into two components since it is a properly 
imbedded disk. 
A subset $C$ in $\Lspace$ is in the closure of a component $A$ of 
$\Lspace \setminus \mathcal{C}(p_1, \vecu_1)$ in 
the {\em direction of} $\vecu_1$ (resp. $-\vecu_1$) if 
the component $A$ is one given in the direction of $\vecu_1$ (resp. $-\vecu_1$) from 
an interior point of the upper cone of $\mathcal{C}(p_1, \vecu_1)$.  

For the following, we allow $\vecu_1$ and $\vecu_2$ to be parallel. 
\begin{proposition}\label{prop:gendisj}  
Suppose that $\vecu_1, \vecu_2$ are consistently oriented, noncrossing space-like vectors. 
Then the following are equivalent\/{\em :} 
\begin{itemize}
\item The crooked planes $\mathcal{C}(p_1, \vecu_1)$ and 
$\mathcal{C}(p_2, \vecu_2)$ are quasi-disjoint.
\item $p_2 - p_1 \in \clo((D(\vecu_2) - D(\vecu_1))$.
\item The closure of a component of $\Lspace \setminus \mathcal{C}(p_i, \vecu_i)$
contains $\mathcal{C}(p_{j}, \vecu_{j})$ for $i\ne j, i, j = 1, 2$.
\end{itemize} 
\end{proposition} 
\begin{proof} 
First, we prove when $\vecu_1$ and $\vecu_2$ are not parallel. 
The first item implies the second and third items by Theorem \ref{thm:disj}  
since we can take the limits. 

The second item implies the first one since from any point of a convex polytope 
$\clo((D(\vecu_2) - D(\vecu_1))$, 
we can find an arbitrarily close point in 
$\mathrm{int}(D(\vecu_2) - D(\vecu_1))$. 
Hence, we can find a sequence of disjoint crooked planes converging to 
the desired ones.

In the situation of the third item, assume $i=1, j= 2$, without the loss of generality.
Suppose that $\mathcal{C}(p_2, \vecu_2)$ is in the closure of a component of 
$\Lspace \setminus \mathcal{C}(p_1, \vecu_1)$ in the direction of $\vecu_1$. 
Let $\vecu'_2$ be a consistently oriented spacelike vector ultraparallel to $\vecu_2$ 
so that $(\vecu'_2)^\perp \cap \Ss_+$ is 
in a component of $\Ss_+ \setminus (\vecu_2)^\perp$ different from that 
of $\Ss_+ \cap (\vecu_1)^\perp$. 
Then we can make a new crooked plane $\mathcal{C}(p'_2, \vecu'_2)$ by 
choosing $p'_2$, $p'_2\ne p_2$, so that $p'_2 - p_2 \in \mathrm{int}(D(\vecu'_2) - D(\vecu_2))$. 
Then $\mathcal{C}(p'_2, \vecu_2')$ is in a component of 
$\Lspace \setminus \mathcal{C}(p_2, \vecu_2)$ in 
the direction of $-\vecu_2$. 
Hence, it is also in the component of  
$\Lspace \setminus \mathcal{C}(p_1, \vecu_1)$ in the direction of $\vecu_1$. 
We can take a sequence of such pairs $(\vecu'_{2, i}, p'_{2, i})$ converging 
to $(\vecu_2, p_2)$ so that 
$p'_{2, i} - p_2 \in \mathrm{int}(D(\vecu'_{2, i}) - D(\vecu_2))$. 
Hence, we obtain the first item. 
%

Suppose now that $\vecu_1 = \pm \vecu_2$.  Then we do the same argument for 
perurbed $\vecu'_2$ so that it is no longer parallel, and we take the limiting facts as 
$\vecu'_2 \ra \vecu_2$. 
\end{proof} 

Given three crooked planes $P_1, P_2$ and $P_3$, we say that 
$P_2$ {\em quasi-separates} $P_1$ and $P_3$ 
if the closure of a component of $\Lspace \setminus P_2$ containes $P_1$ and 
the closure of the other component of $\Lspace \setminus P_2$ contains $P_3$. 

\begin{lemma}\label{lem:quasi-sep} 
	Let $P_i, i=1, 2, 3$ be crooked planes so that $\vecu_i, i=1,2,3,$ are not parallel. 
Suppose that $P_2$ separates $P_1$ and $P_3$.  	
Then we can deform $P_i, i=1, 2, 3$ by continuous parameters 
so that the resulting disks $P'_i, i=1, 2, 3$, are mutually
disjoint and $P'_2$ separates $P'_1$ and $P'_2$.
	\end{lemma} 
\begin{proof} 
We choose the vectors so that  $\vecu_1, \vecu_2$ to be consistently oriented
and so are $\vecu_2, \vecu_3$. 
We have  by Proposition \ref{prop:gendisj} that   
$p_1 - p_2 \in \clo((D(\vecu_1) - D(\vecu_2))$
and $p_3 - p_2 \in \clo((D(\vecu_3) - D(\vecu_2))$.
Then we can choose $p_i(x)$, $i=1, 3$, so that  $p_i(t), i=1, 3$ is in the interior of the correspondiing polytopes for $t> 0$.  We choose $P_i(t)$ to be given by 
$\mathcal{C}(\vecu_i(x), p_i(t))$ for $t >0$. 
	\end{proof}

\subsection{Crooked planes and the $1$-complex.} \label{sub:genconj} 
We will now prove Proposition \ref{prop:genconj} using the limiting arguments
in this subsection: 
\begin{proposition} \label{prop:genconj}
Let $M:=\Lspace/\Gamma$ be a Margulis space-time with parabolics
for a discrete free  group $\Gamma$ of rank $\geq 2$ in $\Isom(\Lspace)$. 
Let $K$ be a bouquet of circles in $M$ homotopy equivalent to $M$. 
Then there exists a collection of immersed crooked planes in $M$ 
with the following properties\/{\em :} 
\begin{itemize} 
\item Any two crooked planes in the collection are quasi-disjoint from one another. 
\item For each edge of $K$, there exists at least one immersed crooked plane meeting it 
odd number of times. 
\end{itemize} 
\end{proposition}

Let $\Gamma$ be a Margulis group in $\Isom(\Lspace)$ of rank $n$. 
We choose a parameter $\Gamma_t, t \in [0, \eps)$ generated by $g_1(t), \dots, g_n(t)$
and without parabolics for $t> 0$ by Theorem \ref{thm:main}. 
For $t=0$, assume that we have a proper action of $\Gamma_t$. 
By Theorem \ref{thm:main}, we know that $\Gamma_t$ acts properly on 
$\Lspace$ for $0 \leq t < \eps$ for some $\eps> 0$. 

By \cite{DGK162}, we know that $\Lspace/\Gamma_t, t>0,$ admis a crooked plane cell decomposition. 
Let $C_1(t), \dots, C_n(t), t>0,$ be the essential collection 
embedded crooked planes in $\Lspace/\Gamma_t$ so that 
the complement of the union is a disjoint union of cells. 

Again, we recall facts from \cite{CDG22}. 
We denote by $\Sigma_{t, +}$ the closure of $\Ss_+/\Gamma_t$ in 
$\Sigma_t$.  Then $\chi(\Sigma_+) = 1 - n = \chi(\Sigma_{+, t})$. 
Note that $\hat M_t:= (\Lspace \cup \Sigma_t)/\Gamma_t, t>0$ is a compact handle body 
of genus $n$ by \cite{CG17}. 
Crooked planes $C_j(t)$ meet the boundary of the handlebody in disjoint collection of 
simple closed curves.

Note that $s(C_i(t)) := C_j(t) \cap \Sigma_{+, t}, j=1, \dots, n, t > 0$ form a disjoint collection of geodesic arcs connecting the boundary components of $\Sigma_{t, +}$. 
Note that $C_j(t)$, $j=1, \dots, n$,
are chosen without regards to continuity with respect to $t$. 

The manifold $\hat M_t:= (\Lspace \cup \Sigma_t)/\Gamma_t$ is a compact handle body. 
The closure of $C_j(t)$ is a disk $\clo(C_j(t))$ with boundary in $\Sigma_t$ for $t> 0$. 

Denote by $\Sigma_{+, t} := \Ss_+ /\Gamma_t$. 
For each $i$, $s(C_i(t)) := \clo(C_i(t)) \cap \Sigma_{+, t}$  
is an embedded arc ending 
at two boundary components of $\partial \Sigma_{+, t}$ in $\Sigma_t$. 
The homotopy class is recorded by 
$[(I, \partial I); (\Sigma_{+, t}, \partial \Sigma_{+, t})], t> 0$.

The inverse image $p_t^{-1}(\bigcup_{j=1}^n C_j(t))$ 
also decomposes $\Lspace$ into $3$-cells. Also note that for any $g\in \Gamma_t$, 
$g\ne \Idd$, $g(C) \ne C$ for any component $C$ here. 

We label these components  $\tilde C_i^j(t)$, $j=1, 2, \dots, \infty$ by distance from the origin
of $\Lspace$. We will denote the closure by $\hat C_i^j(t) \subset \Lspace \cup \tilde \Sigma_t$. 
For $t> 0$, $\hat C_i^j(t) \cap \tilde \Sigma_t$ is homeomorphic to a circle 
that is a union of $\hat C_i^j(t) \cap \Ss_+$ and $\hat C_i^j(x) \cap \Ss_-$ and 
$\zeta(x_1)\cup \zeta(x_2)$ for endpoints $x_1, x_2$ of $\hat C_i^j(x) \cap \Ss_+$. 
(This is called ``crooked circles'' in \cite{CG17}.)

We find a $1$-complex $K_0$ that is the bouquet of geodesic loops which contains the image of $O$ under the projection.  For each $g_i(0)$, we draw a corresponding geodesic $E_i(0)$ connecting $O$ and $g_i(0)(O)$ and projecting. 
Let $\tilde K_0$ be the inverse image containing $O$.
We may suppose that the images of $E_i(0)$ under $p_0$ are embedded loops and are disjoint from one another by choosing the origin carefully. 

We construct $1$-complex $K_t$ by taking $E_i(t)$ to be the segment connecting 
$O$ and $g_i(t)(O)$ in $\Lspace$. 
By taking $\eps$ sufficiently small, for $t< [0, \eps)$, $E_i(t)$ maps to a geodesic loop 
and mutually disjoint for $i=1, \dots, n$ under $p_t$ 
that is embedded in $\Lspace/\Gamma_t$.  

Since $K_t \ra \hat M_t$ induces an isomorphism of the fundamental groups, 
$K_t \ra \hat M_t$ is a homotopy equivalence. 




We say that a set of embedded crooked planes in 
$(\Lspace \cup \tilde \Sigma_t)/\Gamma_t$ is {\em essential} if 
the boundary of each crooked circle is an essential non-separating simple closed curve in 
$\Sigma_t$, and 
the complement of the union is a connected cell. An essential collection does not contain
any separating crooked plane. 

\begin{lemma}\label{lem:essential}  
For the essential collection of crooked planes in 
$M_t:=(\Lspace\cup \tilde \Sigma_t)/\Gamma_t$, each crooked plane meets $K_t$ essentially; that 
is, it cannot be homotopied away. 
Also, each nonseparating crooked plane meets with at least one circle of $K_t$ 
an odd number of times.  
\end{lemma}
\begin{proof} 
Each circle in the bouquet represents a generator of the fist homology group. 
Now, we use the nondegenerate intersection 
$H_1(M_t;\bZ_2) \times H_2(M_t, \partial M_t;\bZ_2) \ra \bZ_2$ for a handle body $M_t$.  
The $\bZ_2$-intersection number with a non-separating crooked plane gives us a generator of the free abelian group 
$H^1(M_t;\bZ_2) \cong H_2(M_t, \partial M_t;\bZ_2)$. Since each circle in $K_t$ gives us a generator of $H_1(M_t;\bZ_2)$, the final statement holds. 
(See Section VI.11 of \cite{Bredon}.) 
\end{proof}

We may have to rename the lifted crooked planes here: 
Let $\hat K_t = \bigcup_{i=1}^n E_i(t)$.  It contains $O$. 
Since $g_i(t)$ are bounded, $\hat K_t$ is always in a fixed compact subset in $\Lspace$. 
Then $h_t(g(j))(\tilde C_l^j(t))$ 
for some $g(j)$ must meet $\hat K_t$. 
By taking a new choice, we rename one of these closest to $O$ to 
be $\tilde C_i(t)$ for each $t$. 
Then $\tilde C_i(t) \cap \Ss_+$ is a geodesic arc ending at the ideal points of
$\Ss_+$. We denote it by $\tilde s(\tilde C_i(t))$. 
It goes to $s(C_i(t))$ in $\Sigma_{+, t}$. 

A component $C_t$ of 
$\hat M_t - \Gamma_t(\bigcup_{j=1}^n(\tilde C_i(t)))$  
contains $O$. 

Two crooked plane meets {\em transversally} at a point $x$ of an relative interior of 
a stem (resp. or a wing) if the at least one 
tangent plane at $x$  to the stem (resp. a wing) containing $x$ meet transversally with such tangent plane at $x$ of the other one. 

Given a union $D$ of finitely many convex subsets of hyperplanes or affine lines in $\Lspace$, a point $x \in D$ is a {\em flat pointt} of $D$ if there is an open set in  a subspace containing $x$ and contained in $D$.

\begin{lemma} \label{lem:geolimdisk} 
Suppose that we have two sequences $\{D_{1, i}\}$ and $\{D_{2, i}\}$ 
of crooked planes in $\Lspace$ 
so that $D_{1, i} \cap D_{2, i} = \emp$. 
Suppose that $D_{j. i} \ra D_j$ for $j=1, 2$ where $D_j$ is a union of 
finitely many convex subsets of affine subspaces of codimension $\geq 1$. 
Then there can be no transversal intersections of $D_1$ and $D_2$ at any point
that is a flat point of both $D_1$ and $D_2$. 
\end{lemma} 
\begin{proof}
Since the geometric limit does so, any nearby subspaces containing stems or wings must still possess some transversal intersection. 
By taking a subsequence if necessary, we may assume that the sequences of 
cones in stems and the wings of $D_{j, i}$, $j=1, 2$, converges to a convex set in $D_j$. 
 Let $x$ be the intersection point of $D_1$ and $D_2$ in the flat part of $D_1$ and $D_2$.
Let us denote these convex sets by $C_{j, l}$ for $l$ in a finite set $L_j$.
We may assume that the cardinality of $L_j$ is $\leq 4$. 
Then there is a $U_j \subset D_j$ for $j=1, 2$ that is a hyperplane and is a union of finitely many convex sets in $D_j$.  


There exists a subsequence of points $x_i$ in $D_{j, i}$ converging to $x$. (See Section \ref{sec:preliminary}.) We can assume without loss of generality that the tangent planes 
at $x_i$ and for the cones and wings containing it in $D_{j, i}$ or adjacent to  those all 
converges to a single vector space up to a choice of subsequences
since $x$ is a flat point of $D_j, j=1, 2$. 
By the transversality of $D_1, D_2$, the corresponding  tangent planes for $D_{1, i}$ and $D_{2, i}$  are transverse for sufficiently large $i$. 
This means again that the crooked planes $D_{1, i}, D_{2, i}$ meet for some large $I$, a contradiction. 
 

\end{proof} 


%
%

Now we take the sequence $t_i \ra 0$ so that 
$\hat C_{t_i, j}$ converges for each $j= 1,\dots, n$
in the Chabauty topology. 
We denote the limits by $\che C_j$.

 
%



We will show that $\che C_j$ are all crooked planes to 
prove Proposition \ref{prop:genconj}.
We now go over the possibilities of Lemma \ref{lem:divcrook} of $\che C_j$, $j=1, \dots, n$
to eliminate the other possibilities. 

Suppose that at least one $\che C_j$ contains a disk $D$ 
a complete null plane or a complete time-like plane.
Since any two planes which are not parallel must intersect transversally, $2$-dimensional
flat subsets of $g(\che C_j)$ intersects some $2$-dimensional flat subsets of $\che C_j$ transversally for some nontrivial $g\in \Gamma$. 
This cannot happen by Lemma \ref{lem:geolimdisk}.  

Suppose that $\che C_j$ is a union of null half-plane with a time-like half-plane meeting at a null-line. The closure of $\che C_j$ meets $\partial \Lspace$  is 
a union of a great segment of form 
$\zeta(x)$ for $x\in \partial \Ss_+$ and a great segment transversal to $\partial \Ss_+$ ending at $x$ and $x_-$ respectively. 
We call the former one $n_j$ a {\em null great segment}, and 
 the later one $t_j$ a {\em time-like great segment}.
Let $l$ denote the intersection of this with $\Ss_+$. 
The union of  two segments bounds a lune $L_i$ in $\Ss$. 

By geometry, 
one of the following holds for $\gamma \in \Gamma - \{\Idd\}$ not fixing a vertex of $L_i$: 
\begin{itemize}
\item $\gamma(n_i) \cap n_i \ne \emp$.
\item $\gamma(t_i) \cap t_i \ne \emp$. 
\end{itemize} 
Thre are many such $\gamma$s. 
  
The boundary of the corresponding lune must intersect transversally in interiors of 
the segment since a lune cannot be a subset of another lune unless the vertices coincide. 

Since these great segments are in the boundary components 
of some flat disks in $\Lspace$, 
this implies that $\che C_j$ intersects its image transversally in flat points. 
This is a contradiction by Lemma \ref{lem:geolimdisk} again.

%
%

The next possibility for $\che C_j$ is also ruled out below: 

\begin{lemma} \label{lem:homdouble} 
A component of $\che C_j$ cannot be a homologically doubled null half-plane.
\end{lemma}
\begin{proof}
In this case, the component $D_t$ meets $K_t$ even number of times 
for some $t> 0$ homologically. 
This contradictions by Lemma \ref{lem:essential}. 
\end{proof}

Hence, we can have as the Chabauty limit the crooked planes by
Lemma \ref{lem:divcrook}. 
This completes the proof of Proposition \ref{prop:genconj}.

Given a sequence of crooked planes, we can always obtain sequences of corresponding stems, wings, and origins. 
A sequence of crooked planes in $\Lspace$ {\em  degenerates} if 
no subsequence converges to a crooked plane. 
(See Lemma \ref{lem:divcrook}.)

\begin{corollary} \label{cor:nodegeneracy} 
There is no degerating sequence of quasi-injective immersed 
crooked planes in $\Lspace$ meeting a compact set in $\Lspace$ mapping to 
a quasi-injective immersed crooked planes in $M$. 
This also implies that the set of origins are in a bounded subset of $\Lspace$. 
\end{corollary}
\begin{proof} 
Suppose not. Then we can find a sequence of crooked planes. We extract a convergent subsequence and the limit. 
Then we use Lemma \ref{lem:divcrook} and the arguments in the last part of Section \ref{sub:genconj} in proving the Proposition \ref{prop:genconj} to obtain contradiction. 
\end{proof}

%




\section{The quasi-disjointness of the crooked disks} \label{sec:crookeddisj}

\subsection{Properties of the crooked planes} \label{sub:procrooked}


%

For each immersed crooked plane $P$ in $M:=\Lspace/\Gamma$, we can lift it to a crooked plane 
$\tilde P$ in $\Lspace$.  Define $\tilde s(\tilde P)$ to be 
$\clo(\tilde P) \cap \Ss_+$ and $s(P)$ its image in $\Sigma_+$. 


\begin{lemma} \label{lem:circlecrook}
 Suppose that a crooked plane $C$ immerses quasi-injectively in $\Lspace/\Gamma$. 
\begin{itemize} 
\item 	There is no crooked plane $C$ in $\Lspace$ such that for $g\in \Gamma -\{\Idd\}$,  
	$g$ acts on $\tilde s(C)$. 
\item 	$\tilde s(C):= \clo(C)\cap \Ss_+$ 
	goes to a nonclosed injective geodesic in $\Sigma_+$. 
\end{itemize} 
\end{lemma} 
\begin{proof} 
	Suppose not. Let $C$ be one where $s_C$ has an element 
	$g\in \Gamma -\{ \Idd \}$ acting on $l$. 
	
	Suppose that $g$ acts on $C$ itself. Then $g$ acts on the origin of $C$, which is 
	a contradiction. 
	
	Suppose that $g(C)$ is not $C$. The associated unit space-like vector $\vecu_C$ 
	is an eigenvector of $\mathcal{L}(g)$. 
	Hence $\vecu_C^+$ and $\vecu_C^-$ are also $\mathcal{L}(g)$-invariant. 
	The quasi-disjointness 
	implies that $p_{g(C)} - p_C = g(p_C) - p_C$ is in the span of 
	$\vecu_C^+$ and $\vecu_C^-$ by Proposition \ref{prop:gendisj}. 
	But $g$ acts on a space-like affine line parallel to $\vecu_C$ perpendicular 
to $\vecu_C^\perp$. 
	Therefore, $g(p_C) - p_C$ must have a nonzero $\vecu_C$-component. 
	This is a contradiction. 
The nonclosedness of the second part follows from the first part. 
\end{proof} 

Recall the theory of geometric laminations on hyperbolic surfaces. 
A geometric lamination is a union of complete geodesics mutually disjoint from 
one another. A leaf is a complete geodesic in it. 
A {\em limit leaf} of a set $A$ of leaves is a complete 
geodesic that is in the closure of the union of leaves in  
$A$ but not in $A$ itself.  

Let $s(A) = \{s(P), P\in A\}$ for a subset $A \subset \mathcal{CR}(\Lspace)$. 
\begin{lemma}\label{lem:no_separation} 
Let $A$ be the set of crooked planes in $\Lspace$ mapping to a set of mutually quasi-disjoint quasi-injectively immersed crooked planes in $M$. 
A limit leaf $l_\infty$ of a leaf $l$ of $s(A)$ does not separate 
a leaf $l'$ and $l$ in $s(A)$
unless $l_\infty$ is a simple closed curve parallel to boundary. 
\end{lemma} 
\begin{proof} 
For each $P\in A$, 
$s(P)$ injective and not closed by Lemma \ref{lem:circlecrook}. 
Hence, $s(A)$ forms the set of geodesics which are mutually disjoint and are not closed. 
Suppose that $l_\infty$ is not a simple closed curve not parallel to boundary, 
and separates $l$ and $l'$. 
By taking a subsequence, 
we may assume that there exists a sequence of 
geodesics $l_i = g_i(l), g_i\in \Gamma$ 
with the property that $l_i \ra l_\infty$ and 
$l_i$ separates $l$ and $l'$.
We may assume that $g_i$ form a nonbounded sequence in $\Gamma$
since $g_i$ does not act on any of $l_i$ or $l$. 

Let $C$ be the crooked plane corresponding to $l$, and 
let $C'$ be one for $l'$. 
Since $l_\infty$ separates $l$ and $l'$, so do $l_i$ for sufficiently large $i$. 
 $C_i$ quasi-separates $C$ and $C'$ for $i > I_0$ for some $I_0$. 
We take a compact arc $K''$ from $C$ to $C'$. 
By Lemma \ref{lem:quasi-sep}, 
$C_i$, $i > I_0$, always meets a compact neighborhood $N(K'')$  of the arc $K''$. 
Hence, $C_i \in \mathcal{CR}(\Lspace)_{N(K'')}$ for $i> I_0$. 




By quasi-injectivity, we can deduce 
$p_{g_i(C_i)} - p_{C_i} \in \clo(D(\mathcal{L}(g_i)(\vecu_i) + D(\vecu_i))$
by Proposition \ref{prop:gendisj}  
since we need to reverse one of the vectors for consistency. 
By Corollary \ref{cor:nodegeneracy}, $\{p_{C}| C\in \mathcal{CR}(\Lspace)_{N(K'')}\}$ is in 
a bounded subset of $\Lspace$.  

By Corollary 4.9 of \cite{CDG22}, $g_i(p_{C_i}) = p(g_i(C_i))$ 
cannot be in a bounded subset of  a compact set in $\Lspace$
since $g_i$ is an unbounded sequence. 
This is a contradiction. 
\end{proof}

An embedded geodesic in $\Sigma_+$ lies in $\Sigma_+^\ast$ can meet a horospherical cusp neighborhood or 
hypercyclic neighborhood.

We say that an imbedded geodesic in $\Sigma_+$ {\em returns} on left (resp. right)
during the interval $(t_0, t_1)$, $t_1 > t_0$ if $l(t_1)$ can be reached from $l(t_0)$  by a
transversal  embedded geodesic 
segment disjoint from $l(t_0, t_1)$ and starting in the local left (resp. right) side of $l$ at $t_0$. 

A geodesic $l$ in $\Sigma_+$ {\em spirals} to a simple clossed geodesic $c$ if 
we can find an increasing sequence $t_i$ so that $l(t_{i+1})$ returns to $l(t_i)$ from left or right and the accumulation points of $l(t_i)$ are all in $C$. 
The is equivalent to the condition that a ift of $l$ ends at an end point of a lift of $c$. 

A geodesic $l$ {\em spiral} to a cusp $c$ if 
for each horocyclic neighborhood $B$ of $c$, there is $T$ so that 
$l(t)$ never leaves $B$ for $t> T$. 
This equivalent to the condition that 
its lift ends in a fixed point of the holonomy of $c$. 

Recall that a boundary neighborhood 
is of an end of $\Sigma_+$ is either a horospherical cusp 
neighborhood or a hypercyclic neighborhood union the geodesic end-neighborhood. 

The following is a standard fact perhaps not written up in this form. 
\begin{lemma}[Thurston]\label{lem:one-sided}  
Any embedded geodesic $l$ in $\Sigma_+$ must meet a boundary neighborhood $N$
of a boundary component or cusp end denoted by $c$   
in one of the following way: 
\begin{itemize} 
\item  It does not meet some boundary neighborhood $N'$ inside $N$. 
\item  It spiralling to $c$ and $c$ is a simple closed curve parallel to boundary 
or $c$ is a cusp. 
\end{itemize} 
Furthermore, $N'$ can be chosen depending on the length of $c$.
	\end{lemma} 
\begin{proof} 
%
Consider first when $c$ is not a cusp. 
Let $l_c$ denote the geodesic in $\Ss_+$ covering $c$. 
We can see from $\Ss_+$ that a geodesic meets a hypercyclic region $R$ for $l_c$ in a connected subsegment or it ends at the point where the closure of hypercyclic region meets $\partial \Ss_+$. 

For any such geodesic segment $s$ with both points of $\partial s$ in 
the same component of $\partial R$, 
the minimal distance from $s$ to $l_c$ has a uniform lower bound depending only on
the length of $c$: Otherwise, we can act by the deck transformation corresponding to $l_c$ so that $s$ meets with its image. 
(Actually, we can do the precise lower bound.) 

We now choose $N'$ by this bound. 
When $c$ is a cusp, we can work similarly on the upper half-space model. 
	\end{proof} 
Here, $N'$ depends on the length of $c$ and $N$.

\begin{corollary} \label{cor:bounded}
	We can choose a union $\hat E$ of boundary neighborhoods of $\Sigma_+$
	with the following properties for each quasi-injective immersed crooked plane $P$ in 
$\Lspace/\Gamma$ {\em :} 
	\begin{itemize}
		\item $s(P)$ is disjoint from $\hat E$, 
		\item  or $s(P)$ enters some component of $\hat E$ and ends at the associated ends; that is, 
		it ends at a cusp or spiral along a simple closed geodesic parallel to boundary. 
	\end{itemize}
Furthermore, $E$ can be chosen to be the union of $\eps$-neighborhoods of a 
boundary neighborhood or $\eps$-thin part of the cusps where $\eps$ is depends only on the lengths of the closed geodesics for a fixed $\eps >0$. 
\end{corollary} 
\begin{proof} 
The two items are immediate from Lemma \ref{lem:one-sided}.  
We need to do some quantification of the proof, which is straightforward. 
%
%
\end{proof}

\subsection{The isolation of crooked planes} \label{sub:isolated} 

We now try to understand the property of the crooked disks obtained by taking the limits.


Let $\bdd$ denote the spherical metric on $\SI^3$. 
We denote by $\bdd^H$ the Hausdorff distance on the set of subsets of $\SI^3$. 
We will also use the metric $d_{\Uu \Ss_+}$ on $\Uu \Ss_+$ induced from 
the hyperbolic metric $d_{\Ss_+}$.  The induced metric on $\Uu \Sigma_+$ is 
denoted by $d_{\Uu \Sigma_+}$. 

\begin{lemma} \label{lem:noinfinite} 
Let $A$ be a collection of 
crooked planes in $\Lspace$ meeting a fixed compact set $K \subset \Lspace$
that are quasi-disjoint and maps to quasi-injectively immersed crooked planes 
in $\Lspace/\Gamma$. Suppose that $s(A)$  cuts the surface $\Sigma_+$ into 
a union cells. 
Then if the set $s(A)$ has a leaf infinitely long in an end, 
then the leaf ends at a cusp corresponding to the end,
or spirals along a simple closed geodesic parallel to boundary of $\Sigma_+$ corresponding 
to the end. 
\end{lemma} 
\begin{proof} 
%
Suppose that $s(A)$ does not end at a cusp at an end or spiral around 
a closed geodesic parallel boundary.
Then $s(A)$ is infinitely long outside $\Sigma \setminus \hat{E}$.
We let $\Lambda_1$ to denote the closure of $s(A)$ in $\Uu\Sigma_+$
which maps to $\Sigma_+ \setminus \hat{E}$
by Corollary \ref{cor:bounded}. $\Lambda_1$ can be considered a geodesic lamination. 

Let $\tilde \Lambda_1$ denote the inverse image of $\Lambda_1$ in $U\Lspace$. 
Two leaves of $\tilde \Lambda_1$ are equivalent if there is an element $g\in \Gamma$ sending 
on leaf to the other. 

Since crooked planes in $A$ meets a compact set $K$, we may assume that 
each leaf of $\tilde \Lambda_1$ is $s(P)$ for some crooked plane $P$ 
by taking a limit of crooked planes in $A$ by Corollary \ref{cor:nodegeneracy}.

By our choice, there exists a crooked planes $C_i$  
so that $\clo(C_i) \cap \Ss_+ = l_i$ and the closure of $l_i$ in $\Ss_+$ contains a limiting 
geodesic $l_\infty$ passing a compact set $K'$. 
Let $\vecu_i$ denote the space-like unit vector associated with $C_i$. 
We may also assume that $l_i \ra l_\infty$ for a complete geodesic $l_\infty$ passing $K'$.
Unless $l_\infty$ covers a simple closed geodesic parallel to boundary, 
we must have some element of the image of $l_i$ under a deck transformation element in both components of $\Ss_+ \setminus l_\infty$. 
By Lemma \ref{lem:no_separation},  this is a contradiction. 
Hence, $l_\infty$ cover a simple closed curve parallel  to boundary. 
This is the spiraling case.



%
%
\end{proof}





Now, choose the union of boundary neighborhoods as in Corollary \ref{cor:bounded}. 
A quasi-injective immersed crooked plane $C$ {\em passes} a boundary component or 
a cusp if $s(C)$ spirals into one of the boundary neighborhoods of the  boundary component or ends at a cusp. 
This can happen for both ends of $s(C)$ or just one of the ends of $s(C)$. 

To study this, we choose an embedded solid torus associated with 
a component of $\hat E$ by choosing the component of 
$\hat E$ sufficiently small 
whenever the component of $\hat E$ is a cusp neighborhood. 
When the component of $\hat E$ is not a cusp, we choose a solid torus  that is a compact tubular neighborhood of 
one of $A_i$
We take these so that their union $T_{\hat E}$ satisfies
$T_{\hat E} \cap \Sigma_+ = \hat E$. (See  Section \ref{sub:before}.)
Let $M_{\hat E}$ denote $M \cup \Sigma \setminus T_{\hat E}^o$. 
$\partial M_{\hat E}$ is a union of a compact surface $\Sigma_{\hat E}:= \Sigma \setminus T_{\hat E}^o$ and 
a union ${\mathfrak{A}}_{\hat E}:= (M \cap T_{\hat E})$ of annuli.
$\Sigma_{\hat E} \cap {\mathfrak{A}}_{\hat E} = \partial {\mathfrak{A}}_{\hat E} = \partial \Sigma_{\hat E}$ which is a union of mutually disjoint circles. These are paired by 
components of ${\mathfrak{A}}_{\hat E}$. 

To explain more,
we recall the annulus $A_i$ for each boundary component $a_i$ of $\Sigma_+$. 
Then we take a solid torus $T_i$ a component of $T_{\hat E}$ that is the compact tubular neighborhood of $A_i$.
The boundary of $T_i$ equals $A_i \cup A'_i$ where 
$\mathcal{A'}_i$ is a compact annulus in $M^o$ with boundary a simple closed curve in $\Sigma_+$ parallel to $a_i$ 
and another one in $\Sigma_-$.


Let $E$ denote the union of parabolic components of $\hat E$.
We denote by $T_E$ the union of components of $T_{\hat E}$ associated with components of $E$. 
If $s(C)$ for a crooked plane does not end in a component of $E$ or 
spiral to a simple closed curve parallel to boundary, 
we will call $C$ a {\em safe crooked plane}.
These have compact closures in $M$ as disks disjoint from $E$.  
The safe crooked planes can be isotopied into $M_E$ disjoint from 
$T_E$. 

Otherwise, one case is that $\partial C \cap T_E$ is not empty, 
and $\partial (C - T_E^o)$ is a union of arcs in $\Sigma_E$ and ${\mathfrak{A}}_E$ with end points in $\Sigma_+ \cap {\mathfrak{A}}_E$ and 
$\Sigma_- \cap {\mathfrak{A}}_E$. 

The arcs in ${\mathfrak{A}}_E$ are called {\em cusp hanged arcs.}
Recall that the crooked plane here are called a  cusp hanging one. 
The image of the crooked plane in $M$ is also called the same. 

Another case is when $s(C)$ spirals along a simple closed geodesic parallel to boundary. 
In this case. recall that $C$ is called a{\em boundary spiraling one}, 
and so is the image in $M$. 
Note that $C$ may be cusp hanging in one end of $s(C)$ and boundary spiralling 
at the other end. 

\label{hanging}




Let $\Sigma_*$ denote the closure of the component $\Sigma_+ \setminus \hat{E}$ with boundary parallel closed geodesics removed with the negative Euler characteristic. 

Let $C$ be a crooked plane quasi-injectively immersed in $M$.  
We may assume that $C \setminus T_{\hat E}$ is still isotopic to $C$ by isotopying $C$ in $M$ by isotopying 
$C$ away from $\partial M$. 



Given a geodesic $l$ on $\Sigma$ 
which ends in $\hat E$, we can remove end neighborhoods so that it ends in 
$\hat E)$.  Its homotopy class is uniquely determined 
in $(\Sigma_+, \hat E)$. 

A crooked plane can be oriented by orienting each piecewise totally geodesic subspaces 
so that their boundary orientations disagree on the intersections. 

Also, consider an immersed oriented crooked plane $C$ in $M$, we can 
take the closure in $M\cup \Sigma$ 
with an immersed geodesic segment in $\Sigma_+$
We remove the part of $C$ going outside a compact neighborhood of $M_{\hat E}$, and 
obtain a map $(D,\partial D) \ra (M \cup \Sigma, \Sigma_{\hat E} \cup T_{\hat E})$
for a disk $D$ with an orientation agreeing with that of $C$. 
The homootpy class is well-defined independent of the choices. 

\begin{corollary}\label{cor:homotopy} 
Choose $\hat E$ by Corollary \ref{cor:bounded}. 
Let $H_D$ be the set of homotopy classes  of quasi-injectively immersed oriented 
crooked planes in 
$(D, \partial D) \ra (M\cup \Sigma, \Sigma_{\hat E}  \cup T_{\hat E})$, 
and $H_s$ the homotopy classes of embeddings 
$(I, \partial I) \ra (\Sigma, \hat E)$.
Then there is an injective map
$H_D \ra H_s$ induced by $[C] \mapsto [s(C)]$. 
\end{corollary}
\begin{proof} 
	Given two crooked planes $C_1, C_2$ where 
	$s_i = s(C_i), i=1, 2$, suppose that $s_1$ and $s_2$ cut off appropriately near
the ends are homotopic in 
$(\Sigma_+, \hat E)$.
This means that for each end point of $s_1$, there is one of $s_2$ in the same components of 
$\partial S_*$ and vice versa. Also, there is a parameter $s(t)$, $t\in [0, 1]$, 
of segments ending at the same components of $\hat E$ where $s(0) = s_1$ and $s(1)= s_2$. 

If $s_i$ extends to a geodesic going to the cusp points, then we see that $s_i$ ends in 
a component of $\partial S_*$ that is the boundary component of a cusp neighborhood.
In  this case so does $s_j$ for $j\ne i$ and vice versa. 

Now, $s_i$ determines $\partial \clo(C)\cap \Sigma_+$. 
We lift these to $\mathcal{H}$. 
The endpoints of a lift of $s_i$ determine the segments of form $\zeta(x), x\in\partial \Ss_+$ and determines  $\partial \clo(C) \cap {\mathfrak{A}}_i$ up to isotopy. 
Now, $C$ can be lifted to a crooked plane $\tilde C$ in $\Lspace$. 
Hence, we have a complete homotopy class of $(\clo(\tilde C), \partial \clo(\tilde C))$
in $(\mathcal{H}, \partial \mathcal{H})$. 
By removing appropriate subsets of $\clo(\tilde C)$ and projecting back to $M\cup \Sigma$, 
we have determined the homotopy class in 
$[(D, \partial D); (M \cup \Sigma, \Sigma_+ \cup T_{\hat E})]$. 
\end{proof}

\begin{proposition}\label{prop:LocFinite}  
The closure of the set of mutually disjoint embedded geodesics of form $s(P)$ for quasi-injective 
crooked planes must be  finite outside $E$ provided that 
$A$ has at most one crooked plane in each homotopy class.
\end{proposition} 
\begin{proof} 
Suppose not. Then we can take some compact set $K$ in $\Lspace$ and find 
$\mathcal{CR}(\Lspace)_K$ with infinitely many crooked planes. 
(See Section \ref{sub:genconj} for this choice of $K$.)


The lengths of $(\Ss_+ - \tilde E) \cap C$ for a quasi-injective crooked plane $C$ 
is uniformly bounded by Lemma \ref{lem:noinfinite}
since otherwise we will have some sequence of segments converging to an infinite length leaf in 
$\Sigma_\ast$. 
The homotopy classes $[s, \partial s; \Sigma_*,\partial \Sigma_*]$  when the 
lengths of $s$ are bounded. 
This determines the homotopy class of crooked planes in  
$[D, \partial D; M_E, \Sigma_E]$. 
This means that the homotopy 
classes realized by $C \in \mathcal{CR}(\Lspace)_K$ is finite. 

Since we have only one crooked plane in the homotopy classes,
we have the result. 
\end{proof} 



\begin{proposition}\label{prop:finitehomotopy}
Let $h_t: \Gamma \ra \Isom(\Lspace), t \in [0, 1]$ be a parameter of proper affine representations
whose image contains no parabolic elements  for $t> 0$.
Suppose that $M_t:= \Lspace/\Gamma_t$. Let $C_{i, t}, i=1, \dots, m$ denote 
the set of short crooked planes  cutting $M_t$ into a cell. 
Then there is $t_0, 0< t_0 < 1$ such that 
there is a lift crooked planes $\tilde C_{i, t}$, $0< t\leq t_0$ in $\Lspace$ 
with following properties: 
\begin{itemize} 
\item $\tilde C_{i, t} \ra \tilde C_{i, 0}$ for a short crooked plane $\tilde C_{i, 0}$, $i=1, \dots, m$ 
\item There is a homeomorphism $f_t: M_E \ra M_t$, $0 < t< t_0$ so that $f_t^{-1}(C_{i, t})$ has 
the constant homotopy classes in $(M, T_{\hat E})$  and agrees with that of $C_{i, 0}$.  
\end{itemize} 
\end{proposition} 
\begin{proof} 
For any $t$, $0 < t < 1$, 
we choose  the union of cusp neighborhoods or the boundary neighborhood $E_t$ in $\Sigma_t$ for each $t$ by Corollary \ref{cor:bounded}. 
Denote by $\tilde \partial E_t$ the inverse image of $\partial E_t$ in $\Ss_+$. 
We may assume that $\tilde \partial E_t \ra \tilde \partial E_{t_0}$ whenever $t \ra t_0$. 
Define $\Sigma_{*, t} := \Sigma_t \setminus E_t$. 
$s_{i, t} := \clo(C_{i, t}) \cap \Sigma_{*, t}$ must have uniformly bounded lengths outside $E$
since otherwise we can obtain a subsequence of crooked planes in $\Lspace$ Chabauty converging  to one $C$ at $t = 0$ with infinite $s(C)$.  
The limiting lamination again divides $\Sigma_t$ into cells. 
This contradicts Lemma \ref{lem:noinfinite}. 

Note that the bound on the lengths outside $E_t$ will bound the number of possible homotopy classes of segments ending at $E_t$. 
Since the homotopy classes are then fixed by Proposition \ref{prop:LocFinite}
for $0< t< t_0$ for some $0< t_0$, 
we proved the first and the second items. 

\end{proof}

\begin{proof}[Proof of Corollary \ref{cor:main}] 
Now, $\partial C_{i, t}$, $i=1, \dots, n$, cut $\Sigma$ into a sphere with holes since $C_{i, t}$
cut $M_t$ into a cell. 

The quasi-disjointness and quasi-injectivity of $C_j$, $j=1, \dots, m$ follow by limit arguments with their lifts in $\Lspace$ since $C_{j, t}$ are properly embedded and are disjoint for $j=1, \dots, ,m$.
We may assume that $C_{i, t}$ is short so that $s(C_{i, t})$ are shortest geodesic between 
geodesic end neighborhoods $\Ss_+/\Gamma$ and some choices of 
horocyclic neighborhoods. Hence, 
the limit crooked planes are also short by geometric convergence. 

By Propositions \ref{prop:LocFinite} and  \ref{prop:genconj}, 
we have finitely many crooked planes $C_1, \dots, C_m$ with properties of the conclusions of 
Proposition \ref{prop:finitehomotopy}. These are quasi-disjoint.

We perturb $C_1, \dots, C_m$ so that they meet one another transversally. 
We can use Dehn's lemma and exchanging disks to obtain 
embedded disks isotopic to $C_1, \dots, C_m$ with the identical boundary 
to $\partial C_1, \dots, \partial C_m$ respectively so that 
the complement is a cell:  
By the second item of Proposition \ref{prop:finitehomotopy}, there are only 
finitely many changes of homotopy classes of $C_{i, t}$ and $\partial C_{i, t}$
for $i=1, \dots, m$. Hence, we can perturb the limiting quasi-injective quasi-disjoint 
crooked planes to obtain perturbed disks 
the complement of whose union is a $3$-cell. 

Note that the lengths outside $\hat E$ of $s(C_j)$ are all finite. 
Since $s(C_j)$ has to cut $\Sigma$ into a union of cells, no two of them can be identical. 
The normal vectors to $C_j$ are never parallel since 
$s(C_j)$ are mutually distinct.

Let $C_j$ have orientation induced from the perturbed disks which are given the boundary orientation from the cell. Then $s(C_i)$ also has the induced boundary orientation. 
We denote by $\vecu_i^+, \vecu_i^-$ the endpoints of $s(C_i)$ following the orientation. 

Also, $s(C_j), j=1, \dots, m$ is a boundary of a cell $C_\Sigma$ in $\Ss_+$ since their images 
also cut $\Sigma$ into a cell.

Since they are disjoint in $\Ss_+$ and hence on $\Ss_-$ and on $\Ss_0$, the arcs may coincide. 

It is possible that two wings of some lifting crooked planes $\hat C_j$ and 
$\hat C_{j'}'$ in $\Lspace$ of  $C_j$ and $C_{j'}$ 
may meet on an open subset  of the wings.
The wings are bounded by a affine line in the direction $p$ of $\hat C_j$ and $\hat C_{j'}'$  where $p$ is the common endpoint of $s(C_j)$ and $s(C_{j'})$. 
Then one wing $W_j$  must be a subset of the other $W_{j'}$.  
If it is a proper subset, then the image of one of two stem cones of $\hat C_j$ must intersect transversally with $W_{j'}^{\prime o}$. This contradicts the quasi-disjointness. 
We may still have $W_j = W_{j'}'$. 

We will now show that this happens for only cusp-hanging end point of 
the crooked planes. 

If the null direction of a lift of $W_j$ 
is not a parabolic fixed point, then it must end at an arc $\alpha$ in 
$\partial \Ss_+$.

For $i \ne j$,  $s(C_{i_1, t})$ and $s(C_{i_2, t})$ are in different homotopy classes in
$(\Sigma_+, \hat E)$. Removing all geodesic end-neighborhood from $\Ss_+$, 
these geodesics will end at boundary parallel geodesics or at a cusp. 

These $s(C_{i_1, t})$ and $s(C_{i_2, t})$ will converge to
geodesics that are cusp hanging or not in an end. 
If the limit is not cusp hanging at an end, then they meet the boundary parallel geodesics at 
different point since they are the shortest geodesics between these geodesics and the
horocyclic neighborhoods. 
Hence, their noncusp ending endpoints must be distinct at the boundary parallel geodesics. 

So their wings cannot coincide as above. 
So the wing coinciding happens only for cusp hanging ends. 

Finally the closure of each component of the complement of the crooked disks are contractible since they are deformed from cells as $t\ra0$ and they are retracts of contractible subspaces. These follow since the homotopy classes change only finitely many times.



\end{proof}

\appendix 

\section{The metric space of crooked planes and its completions} \label{app:metricspace}


We introduced a topology on the set of crooked planes on $\Lspace$. 
The closure of each crooked plane meets $\Ss_+$ in a complete geodesic. 
There is an origin of the crooked plane. 
The space is parameterized by the space of unoriented geodesics on $\Ss_+$ 
congruent to $(\partial \Ss_+\times \partial \Ss_+ - \Delta)/\bZ_2$. 

Recall that $\mathcal{CR}(\Lspace)$ is the space of crooked planes with the Chabauty topology.
The space is in one-to-one correspondence with 
$\mathcal{C}(\Ss_+) \times \Lspace$ since the origin of the crooked planes and the geodesic in 
$\Ss_+$ determines the crooked plane. 

For a compact subset $K$ of $\Lspace$, we denote by 
$\mathcal{CR}(\Lspace)_K$ the subset of crooked planes meeting $K$. 

A {\em null strip} is the closed subset of a null plane bounded by two parallel null lines. 
Also, given a crooked plane, we may label its cones $C^e$ for $e=+, -$ according to whether it is positive one or not, assign its boundary null-lines $l^e$ for $l^e, -$
 arbitrarily and  the wings $W^e$ adjacent to $l_e$. 

\begin{lemma} \label{lem:stemswings} 
Given a sequence of crooked planes $D_i$, we can extract a corresponding sequence of 
wings $W^e_i$ for $e=+, -$, the corresponding sequences of cones $C^e_i$ for $e=+, -$,
the sequence $l^e_i$ for $e=+, -$ where $l^e_i$ is the segment bounding $W^e_i$ 
and the sequence of origins $p_i$. 
 We can choose a subsequence for $\{D_i\}$ so that 
the following hold\/{\em :} 
\begin{description} 
\item[Wing limit] $\{W^e_i\}$ for each $e=+, -$  converges to 
a strongly null-half planes or a null-half plane or the empty set.
\item[Cone limit]   $\{C^e_i\}$   for each $e=+, -$  converges to a cone, a null-line, a half-null line, a null strip, 
a half-time like plane bounded by  a null-line, a strongly null-half plane, 
a null plane, a time-like plane, or the empty set. 
\item[Line limit] $\{l^e_i\}$  for each $e=+, -$  converges to a null line or the empty set. 
\item[Origin limit] $\{p_i\}$ converges to a point or the empty set.
\end{description}  
In the cone limit null-strip case, the sequence of one of the cones converges to a null strip and the sequence of other ones converges to the empty set. 
\end{lemma} 
\begin{proof} 
We first extract the subsequence 
where the corresponding sequence of origins and the corresponding lines converge, and 
We form sequences of planes from the sequence of stems and the sequence of 
corresponding wings with same $e$s. 
We assume that each of these sequences
converges to a time-like or a null-like plane or the empty set
by Lemma \ref{lem:convS}. 

The last two  items are clear. 

For the cone case, 
we consider the sequences of the origins and the bounding lines. 
We use Lemma \ref{lem:conv} now: 
If these sequences converge to the empty set, the limit is either the time-like plane or 
the null-plane or the empty set correspondingly. 
If only one of $\{l^e_i\}$ converges to a null-line, then the limit is either a time-like plane bounded by the null-line or a null-plane bounded by the null-line. 
If both of $\{l^e_i\}$ converge to a common null-line, then the limit is a null-line or a half-null line dpending on whether $\{p_i\}$ converges or not.  
If both of $\{l^e_i\}$ converge to  parallel but distinct null-lines respectively, 
then the limit is a null strip. 
If both of $\{l^e_i\}$ converge to non-parallel but distinct null-lines respectively, 
then the limit is a cone.  

In the wing case, if the sequence of bounding lines converges to the empty set,
then the limit is a null-plane. If the sequence of bounding lines converges to a null-line, then 
the limit is a strongly null-half plane. 
\end{proof}


Any embedded disk $D$ in $\clo(\Lspace)$ has a $2$-cycle representing 
the generator of $H_2(D, \partial D; \bZ_2)$. 
When we have  a $2$-cycle $z$ in $\clo(\Lspace)$ on a set $L$ in $\clo(\Lspace$) and $x \in L$, 
we can consider the map 
$H_2(L, L \cap \partial \Lspace; \bZ_2) \ra H_2(L, L-\{x\};\bZ_2)$  induced by $L \ra (L, L-\{x\})$.
We say that $z$ homologically double-covers a subset $A \subset L$ if for almost all points of 
$x\in A$, the image of $z$ in $H_2(L, L-\{x\};\bZ_2)$ is $2$ times a generator. 
%
In the third case,  the pair $(L, z)$ is called the {\em doubled null-plane}.

Notice that except for the fourth case, each set separates $\Lspace$ into two components. 
For the proof of the following Lemma, we will use the fact: 

\begin{lemma}\label{lem:finite}
 If we have a sequence of 
closed set $L_i$ that is a union of finitely many convex sets $L_{i, j}, j=1, \dots, m$,
and each $L_{i, j} \ra L_{\infty, j}$ for each $j$. Then 
$L_i \ra \bigcup_{j=1}^m L_{\infty, j}$. 
\end{lemma}
\begin{proof}
Again, this is a simple consequence of 
Lemma  E.1.2 of \cite{BP92}.
\end{proof} 

We extend Lemma \ref{lem:stemswings}: 
\begin{lemma}\label{lem:divcrook}
Given any sequence of elements of  
$\mathcal{CR}(\Lspace)_K$, we can find a subsequence
so that its corresponding sequence of wings and cones converge and the sequence itself 
 converging to one of the following
in the Chabauty topology:
\begin{itemize} 
\item A complete time-like plane meeting $K$. 
\item  A complete null plane meeting $K$
\item A strongly null half-plane meeting $K$.
\item  A time-like half-plane union with a strongly null half-plane meeting meeting on a null line. 
\item A crooked plane.
\end{itemize} 
In the third case, we give a $2$-cycle representing the generator of 
$H_2(\clo(C), \partial C;\bZ_2)$ for each crooked plane. 
The convergences yield that the corresponding subsequence of $2$-cycles representing 
the closure of the crooked planes in $\clo(\Lspace)$ 
converges to a the $2$-cycle that homologically double covers the 
strongly null half-plane and represent a zero-class in 
$H_2(\clo(N) \cup \partial \mathcal{H}, \partial \mathcal{H};\bZ)$.
\end{lemma}
\begin{proof}
Given any sequence of elements $C_i$  of $\mathcal{CR}(\Lspace)_K$, 
either a wing $W^e_i$, $e=+, -$ or a cone $C^e_i$, $e=+, -$ 
of $C_i$ meets $K$.  Let $l^e_i, e = +, -$ be the two null lines of $W_i$ 
bounding the two cones. 
Let $p_i$ denote the intersection point of $l^+_i$ and $l^-_i$. 

We assume that the sequences of vector spaces parallel to $W_i, C^e_i, l^e_i$ 
are convergent by choosing subsequences by Lemma \ref{lem:stemswings}.  


We define the following cases: 
\begin{itemize} 
\item[$(p)_{(1)}$] $p_i$ is converging to a point of $\Lspace$. 
\item[$(p)_{(3)}$] $p_i$ exits any compact subset of $K$. 
\item[$(l)_{(1)}$] $l^e_i$ converges to a null-line $l^e_\infty$ for one of 
$e= +, -$,  and the sequence for the other $e$ 
exits any compact subset of $\Lspace$ 
\item[$(l)_{(2)}$]  $l^e_i$ converges to a null-line $l^e_\infty$ for two of 
$e= +, -$. 
\item[$(l)_{(3)}$] $l^e_i$ for both $e=+, -$ 
becomes disjoint from any compact subset of $\Lspace$ 
for sufficiently large $i$. 
\item[$(w)_{(1)}$] The sequence of null-planes containing $W^e_i$ converges 
to a null-plane in $\Lspace$ for one of $e=+, -$. 
For other $e$, the sequence of the null-planes exits any compact set. 
\item[$(w)_{(2)}$] The sequence of null-planes containing $W^e_i$ converges 
to a null-plane in $\Lspace$ for both $e=+, -$. 
\item[$(w)_{(3)}$] The sequence of null-planes containing $W^e_i$ exits any
compact subset in $\Lspace$ for both $e=+, -$. 
\item[$(c)_{(1)}$] The sequence of time-like planes containing $C^e_i$ converges 
to a time-like or a null plane in $\Lspace$ for both of $e=+, -$. 
\item[$(c)_{(3)}$] The sequence of time-like planes containing $C^e_i$ exits any
compact subsets in $\Lspace$ for both $e=+, -$. 
\end{itemize}
Note that 
\begin{itemize} 
\item $(l)_{(2)}, (w)_{(2)}$ and  $(c)_{(2)}$ are weaker than  $(p)_{(1)} $, 
\item  $(w)_{(k)}$ and $(c)_{(k)}$ are weaker than $(l)_{(k)}$ for each $k= 1, 2$.
\end{itemize} 
and hence some choices do not occur.

By choose subsequences, we can assume of a choice of one element from
each of  $\{(c)_{(3)}, (c)_{(1)}\}$, $\{(w)_{(3)}, (w)_{(2)}, (w)_{(1)}\}$, $\{(l)_{(3)}, (l)_{(2)}, (l)_{(1)}\}$, 
and $\{(p)_{(3)}, (p)_{(1)}\}$.  These are all possibilities by Lemma \ref{lem:stemswings}. 
We will list the possibilities in the order given here.


We list the possibilities of the Chabauty limits 
where the list in an item ends if the remaining possibilities are determined. 
We extract the subsequence so that 
corresponding sequences of four polygonal regions consisting of the two wings and two cones of the stems converges or leaves every compact subset respectively. 
The Chabauty limit  of each corresponding  sequence of convex polygons is always a convex $1$- or $2$-dimensional domains. 
These follow by Lemmas \ref{lem:finite}, \ref{lem:conv} and \ref{lem:convS}. 
\begin{description} 
\item[$(c)_{(3)}, (w)_{(3)}$] This contradicts the sequence being in 
$\mathcal{C}_K$. 
\item[$(c)_{(3)}, (w)_{(2)}$] Since the time-like plane separates $\Lspace$ into two 
components, we cannot have two sequences of closures of the wings  convergent.
\item[$(c)_{(3)}, (w)_{(1)}$] In this case,  it converges to a null plane where 
a sequence of one of the corresponding wings will converge to the same null plane
and the corresponding sequences of cones and the other corresponding wings will exit any compact sets. 

\item[$(c)_{(1)}, (w)_{(3)}$] The possibilities  are a null-plane
or a time-like plane. This follows since the sequences of the wings exits any compact set and the sequence of one of the cones becomes larger and larger with the 
sequence of their  boundary leaving every compact subset, 
and the sequence of the other cones leaving every compact subset.  
And $(l)_{(3)}, (p)_{(3)}$ follow.
\item[$(c)_{(1)}, (w)_{(2)}$]  $ $ 
\begin{itemize} 
\item One possibility of the limit is a crooked plane where every sequence of 
the corresponding polygonal regions converges. 
\item The other possibility is a strongly 
null half-plane, where $l^+_i$ and $l^-_i$ converge to the null lines 
$l^+$ and $l^-$ in the same direction, and 
the corresponding 
sequences of the wings are converging to one or two strongly null half-planes, and 
\begin{itemize} 
\item the sequence of one of the cones converges 
to the null strip bounded by $l^+$ and $l^-$, and the sequence of other ones leaves every compact subset, 
or 
\item the sequence of the union of cones converges to the null line itself whenever $l^+$ and 
$l^-$ are equal. 
\end{itemize} 
(Note there that the strongly null-half spaces must overlap at the limit by the orientation considerations. That is, these limits do not form a complete null plane.)
\end{itemize} 
\item[$(c)_{(1)}, (w)_{(1)}$]  $ $
\begin{itemize} 
\item One possibility is 
the union of a strongly null half-plane and a time-like half-plane 
meeting at a null line with $(l)_{(1)}$ occurring. 
Here, the sequence of one corresponding  wings converges to a strongly null half-plane,
the sequence of the other corresponding wings leaves any compact subset, 
the sequence of one of the cones converges to a time-like  half-plane, 
and the sequence of the other cones leaves every compact subset. 
\item The other possibility is a strongly null half-plane, where  
the sequence of one corresponding wings is converging to the strongly null half-plane, 
the sequence of other corresponding wings is leaving every compact subset, 
the sequence of one of the cones converges to the same strongly null half-plane, and the sequence of the other cones leaves every compact set.
\end{itemize} 
\end{description}


Strongly null half-planes as limits 
are obtained in the second case of $(c)_{(1)}, (w)_{(2)}$ and the second case of $(c)_{(1)}, (w)_{(1)}$. 
In these cases, the limits are doubled ones. 
We can choose a sequence of finite cycles with bounded number of simplices so that 
it converges to a desired cycle $z$.
Since there are no other cases where the limit is a strongly null half planes,
this proves the final part. 
\end{proof}

\bibliographystyle{plain} 
\bibliography{cdg}

\end{document}